\documentclass{amsart}
\usepackage{amsmath}
\usepackage{amssymb}
\usepackage{amsfonts}
\usepackage{graphicx}
\usepackage{tikz}
\usetikzlibrary{knots}
\usetikzlibrary{hobby}
\usetikzlibrary{decorations.markings}
\usetikzlibrary{shapes.geometric}

\usepackage{longtable}
\usepackage{array}
\usepackage{mathtools}

\usepackage{lipsum} 
\usepackage[T1]{fontenc}
\newtheorem{theorem}{Theorem}

\newtheorem{example}{Example}
\newtheorem{remark}{Remark}

\newtheorem{proposition}{Proposition}

\tikzset{knot_diagram/.style={draw=black, thick, line cap=round}}
\tikzset{tensor_diagram/.style={draw=black, line cap=round}}
\tikzset{
	knot diagram/every strand/.append style={
		thick,
		black
  }
}
\tikzset{square/.style={regular polygon,regular polygon sides=4}}

\title{Electric group for knots and links}
\author{Korablev Ph. G.}

\address{Chelyabinsk State University, Chelyabinsk, Russia; N.N. Krasovsky Institute of Mathematics and Meckhanics, Ekaterinburg, Russia}
\email{korablev@csu.ru}
\date{}

\begin{document}

\begin{abstract}
    In 2014 Andrey Perfiliev introduced the so-called electric invariant for non-oriented knots. This invariant was motivated by using Kirchhoff's laws for the dual graph of the knot diagram. Later, in 2020, Anastasiya Galkina generalised this invariant and defined the electric group for non-oriented knots. Both works were never written and published. In the present paper we describe a simple and general approach to the electric group for oriented knots and links. Each homomorphism from the electric group to an arbitrary finite group can be described by a proper colouring of the diagram. This colouring assigns an element of the group to each crossing of the diagram, and the proper conditions correspond to the areas of the diagram. In the second part of the paper we introduce tensor network invariants for coloured links. The idea of these invariants is very close to quantum invariants for classical links.
\end{abstract}

\maketitle

\section{Introduction}

The main purpose of this paper is to introduce the notion of electric group for oriented knots and links. An electric group is a group that can be assigned to any link, and this group is different from the classical knot group.

Initially, the idea of constructing so-called electric invariant for knots was introduced by Andrey Perfiliev in almost 2014. This work was never published. The electric invariant was based on the application of Kirchhoff's laws to the dual graph of the knot diagram. Later, in 2020, my student Anastasiya Galkina in her bachelor thesis extended this electric invariant to a simple version of the electric group. But this work was never published either. The current approach to the electric group is simpler and more general than the previous one.

Any homomorphism from the electric group of a link to a finite group can be interpreted as a proper colouring of a link diagram. This colouring assigns a group element to each crossing such that the product of the corresponding elements at the boundary of each region of the diagram is trivial. This view is similar to the colouring of diagrams by elements of a finite quandle. 

In the second part of the paper we construct an invariant for diagrams coloured by elements of the finite group, i.e. for the pair $(D, \xi)$, where $D$ is a diagram and $\xi$ is a homomorphism from the electric group to a finite group. The idea of this invariant is close to the idea of quantum invariants for links (see for example \cite{TQI}). It is called a tensor network invariant because it can be computed by replacing the diagram by a set of special tensors and making a contraction of the tensor product of these tensors. The tensors that replace crossings depend on the colour of the crossing. In order to define the invariant in a correct way, a so-called consistent tensor system is required, which consists of several families of tensors, indexed by elements of the group and satisfying a set of axioms. These axioms guarantee that the result does not change under Reidemeister moves.

The structure of the paper is as follows. In section 2 we define the electric group and prove that this group is a correctly defined invariant of oriented links. In section 3 we describe the set of homomorphisms from the electric group to a finite group as colourings of the link diagram. In section 4 we define the class of tensor network invariants for coloured diagrams. In particular, this leads to the invariant of oriented links. In section 5 some questions for further development of the theory are formulated.

\section{Electric group}

Let $D$ be a diagram of the oriented link $K$ on an oriented 2-sphere $S^2$. In the following we will assume that the sphere $S^2$ is oriented counter-clockwise. But it's also possible to choose the opposite orientation.

Let $C(D) = \{c_1, \ldots, c_n\}$ be the set of crossings of the diagram $D$, and $A(D) = \{a_1, \ldots, a_{n + 2}\}$ be the set of complement areas of the diagram $D$ on the 2-sphere $S^2$. Construct a group $\mathcal{E}(D)$ using generators and relators as follows. There are exactly $n + 2$ generators: two generators $a, b$ and $n$ generators, which are in one-to-one correspondence with crossings of the diagram $D$. 

Each area $a_i\in A(D)$ corresponds to the relation $R_{a_{i}}$ of the group $\mathcal{E}(D)$. To write this relation, choose any starting point and walk along the boundary of the are $a_i$ in the direction corresponding to the orientation of the 2-sphere $S^2$. Each crossing $c_j\in C(D)$ in the path corresponds to the word $w_{c_j}$, which is defined as follows. If $x_j$ is a generator corresponding to the positive crossing $c_j$, then $w_{c_j}$ is either $x_j ab$ or $b^{-1}a^{-1} x_{j}^{-1}$ or $a^{-1}x_j a$ or $b^{-1} x_{j}^{-1} b$. It depends on which corner in the neighbourhood of the crossing $c_j$ belongs to the area $a_i$ (figure \ref{Figure:CornersInNeighbourhood} on the left). If the crossing $c_j$ is negative and, as before, $x_j$ is a corresponding generator, then $w_{c_j}$ is either $x_j ab$ or $b^{-1}a^{-1}x_{j}^{-1}$ or $a^{-1}x_{j}^{-1} a$ or $b^{-1} x_j b$. It depends on which corner in the neighbourhood of the crossing $c_j$ is part of the area $a_i$ (figure \ref{Figure:CornersInNeighbourhood} on the right). Then the relation $R_{a_{i}}$ is a product of all words $w_{c_j}$ that we get by going around the boundary of the area $a_i$.

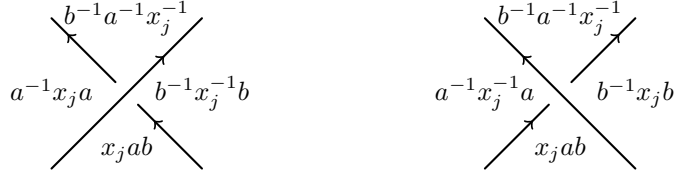
\begin{figure}[ht]
    \begin{tikzpicture}
        \begin{scope}[decoration={
            markings,
            mark=at position 0.75 with {\arrow{>}}}
            ]
            \draw[knot_diagram, postaction={decorate}] (-1, -1) -- (1, 1);
            \draw[knot_diagram, postaction={decorate}] (1, -1) -- (0.15, -0.15);
            \draw[knot_diagram, postaction={decorate}] (-0.15, 0.15) -- (-1, 1);
        \end{scope}

        \draw (0, -0.75) node {$x_j ab$};
        \draw (1.0, 0) node {$b^{-1} x_j^{-1} b$};
        \draw (0, 1.0) node {$b^{-1} a^{-1} x_j^{-1}$};
        \draw (-1.0, 0) node {$a^{-1} x_j a$};
    \end{tikzpicture}
    \hspace{2cm}
    \begin{tikzpicture}
        \begin{scope}[decoration={
            markings,
            mark=at position 0.75 with {\arrow{>}}}
            ]
            \draw[knot_diagram, postaction={decorate}] (-1, -1) -- (-0.15, -0.15);
            \draw[knot_diagram, postaction={decorate}] (0.15, 0.15) -- (1, 1);
            \draw[knot_diagram, postaction={decorate}] (1, -1) -- (-1, 1);
        \end{scope}

        \draw (0, -0.75) node {$x_j ab$};
        \draw (1.0, 0) node {$b^{-1} x_j b$};
        \draw (0, 1.0) node {$b^{-1} a^{-1} x_j^{-1}$};
        \draw (-1.0, 0) node {$a^{-1} x_j^{-1} a$};
    \end{tikzpicture}
    \caption{\label{Figure:CornersInNeighbourhood}Word $w_{c_j}$ in the neighbourhood of the crossing $c_j$ (positive on the left, negative on the right)}
\end{figure}

\begin{example}
    Let $D$ be an oriented minimal diagram of the <<negative trefoil>> knot. For this diagram $C(D) = \{c_1, c_2, c_3\}$ and $A(D) = \{a_1, a_2, a_3, a_4, a_5\}$ (figure \ref{Figure:TrefoilExample}). 

    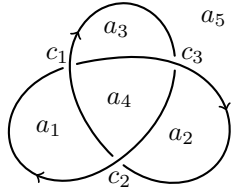
\begin{figure}[ht]
        \begin{tikzpicture}[scale=0.75]
            \draw[knot_diagram] (1.47, 25.888).. controls (1.306, 25.971) and (1.164, 26.103) .. (1.081, 26.267).. controls (0.98, 26.465) and (0.969, 26.702) .. (1.031, 26.915).. controls (1.093, 27.128) and (1.226, 27.316) .. (1.394, 27.462).. controls (1.556, 27.604) and (1.746, 27.704) .. (1.949, 27.777);
            \draw[->, knot_diagram] (3.934, 27.734).. controls (3.925, 27.612) and (3.901, 27.49) .. (3.865, 27.374).. controls (3.722, 26.902) and (3.399, 26.498) .. (3.009, 26.197).. controls (2.718, 25.972) and (2.373, 25.795) .. (2.006, 25.78).. controls (1.822, 25.772) and (1.635, 25.806) .. (1.47, 25.888);
            \draw[->, knot_diagram] (2.83, 26.22).. controls (2.728, 26.318) and (2.634, 26.426) .. (2.55, 26.538).. controls (2.351, 26.802) and (2.193, 27.101) .. (2.118, 27.422).. controls (2.039, 27.761) and (2.059, 28.132) .. (2.23, 28.434);
            \draw[knot_diagram] (2.23, 28.434).. controls (2.316, 28.585) and (2.438, 28.716) .. (2.587, 28.805).. controls (2.737, 28.894) and (2.913, 28.939) .. (3.086, 28.926).. controls (3.251, 28.913) and (3.412, 28.846) .. (3.542, 28.743).. controls (3.673, 28.641) and (3.774, 28.503) .. (3.841, 28.35).. controls (3.889, 28.238) and (3.919, 28.119) .. (3.932, 27.997);
            \draw[->, knot_diagram] (2.199, 27.851).. controls (2.337, 27.886) and (2.479, 27.911) .. (2.619, 27.93).. controls (2.908, 27.97) and (3.2, 27.991) .. (3.49, 27.959).. controls (3.779, 27.928) and (4.066, 27.844) .. (4.311, 27.688).. controls (4.557, 27.531) and (4.758, 27.299) .. (4.851, 27.023);
            \draw[knot_diagram] (4.851, 27.023).. controls (4.897, 26.885) and (4.916, 26.737) .. (4.903, 26.593).. controls (4.889, 26.448) and (4.842, 26.306) .. (4.763, 26.183).. controls (4.645, 25.998) and (4.456, 25.862) .. (4.247, 25.794).. controls (4.037, 25.727) and (3.809, 25.726) .. (3.596, 25.778).. controls (3.388, 25.828) and (3.197, 25.926) .. (3.026, 26.054);

            \draw (2.898, 28.46) node {$a_3$};
            \draw (1.7, 26.712) node {$a_1$};
            \draw (2.96, 27.213) node {$a_4$};
            \draw (4.053, 26.573) node {$a_2$};
            \draw (4.624, 28.661) node {$a_5$};
            \draw (1.854, 28.007) node {$c_1$};
            \draw (4.253, 27.998) node {$c_3$};
            \draw (2.956, 25.798) node {$c_2$};
        \end{tikzpicture}
        \caption{\label{Figure:TrefoilExample}Diagram of the <<negative trefoil>> knot}
    \end{figure}

    By definition:
    \begin{center}
        \begin{align*}
            &R_{a_1} = (x_1 a b)\cdot (b^{-1} a^{-1}x_2^{-1}) = x_1 x_2^{-1},\\
            &R_{a_2} = (x_2 a b)\cdot (b^{-1} a^{-1} x_3^{-1}) = x_2 x_3^{-1},\\
            &R_{a_3} = (x_3 a b)\cdot (b^{-1} a^{-1} x_1^{-1}) = x_3 x_1^{-1},\\
            &R_{a_4} = (b^{-1} x_1 b)\cdot (b^{-1} x_2 b) \cdot (b^{-1} x_3 b) = b^{-1} x_1 x_2 x_3 b,\\
            &R_{a_5} = (a^{-1} x_1^{-1} a)\cdot (a^{-1} x_3^{-1} a)\cdot (a^{-1} x_2^{-1} a) = a x_1^{-1}x_3^{-1}x_2^{-1}a.
        \end{align*}
    \end{center}

    As a result we get 
    \begin{multline*}
        \mathcal{E}(D) = <a, b, x_1, x_2, x_3 | x_1 x_2^{-1}, x_2 x_3^{-1}, x_3 x_1^{-1}, b^{-1}x_1 x_2 x_3 b, a^{-1} x_1^{-1} x_2^{-1} x_3^{-1} a> = \\ = <a, b, x | x^3 = 1>.
    \end{multline*}
\end{example}

\begin{theorem}
    \label{Theorem:ElectricGroup}
    Let $D_1, D_2$ be two diagrams of an oriented link $K$ on an oriented 2-sphere $S^2$. Then the groups $\mathcal{E}(D_1)$ and $\mathcal{E}(D_2)$ are isomorphic.
\end{theorem}
\begin{proof}
    It's known that any two diagrams of the same oriented links are connected by a finite sequence of Reidemeister moves, as shown in the figure \ref{Figure:OrientedReidemeister} (\cite{PRM}). It's enough to prove the theorem in the cases where $D_2$ is obtained from $D_1$ by one of these moves.

    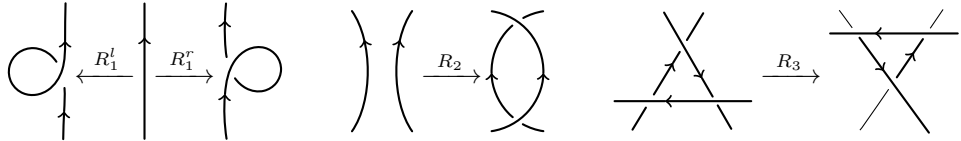
\begin{figure}[ht]
        \begin{tikzpicture}[scale=0.78, baseline={([yshift=-1ex]current bounding box.center)}]
            \draw[knot_diagram] (1.55, 28.873).. controls (1.55, 29.085) and (1.553, 29.297) .. (1.558, 29.51);
            \draw[knot_diagram, ->] (1.402, 28.513).. controls (1.377, 28.55) and (1.348, 28.583) .. (1.316, 28.614).. controls (1.235, 28.69) and (1.128, 28.74) .. (1.017, 28.745).. controls (0.923, 28.75) and (0.828, 28.72) .. (0.752, 28.664).. controls (0.677, 28.607) and (0.623, 28.524) .. (0.602, 28.432).. controls (0.58, 28.34) and (0.593, 28.241) .. (0.636, 28.158).. controls (0.68, 28.074) and (0.754, 28.007) .. (0.842, 27.973).. controls (0.909, 27.946) and (0.984, 27.939) .. (1.056, 27.95).. controls (1.128, 27.96) and (1.197, 27.989) .. (1.257, 28.029).. controls (1.377, 28.111) and (1.46, 28.24) .. (1.502, 28.379).. controls (1.551, 28.538) and (1.55, 28.707) .. (1.55, 28.873);
            \draw[knot_diagram] (1.534, 27.67).. controls (1.541, 27.802) and (1.549, 27.935) .. (1.54, 28.066);
            \draw[knot_diagram, ->] (1.518, 27.193).. controls (1.521, 27.352) and (1.526, 27.511) .. (1.534, 27.67);
        \end{tikzpicture}
        $\xleftarrow{\ R_1^l\ }$
        \begin{tikzpicture}[scale=0.9, baseline={([yshift=-1ex]current bounding box.center)}]
            \begin{scope}[decoration={
                markings,
                mark=at position 0.75 with {\arrow{>}}}
                ]
                \draw[knot_diagram, postaction={decorate}] (0, -1) -- (0, 1);
            \end{scope}
        \end{tikzpicture}
        $\xrightarrow{\ R_1^r\ }$
        \begin{tikzpicture}[scale=0.75, baseline={([yshift=-1ex]current bounding box.center)}]
            \draw[knot_diagram] (1.477, 28.515).. controls (1.48, 28.615) and (1.487, 28.714) .. (1.496, 28.813);
            \draw[knot_diagram, ->] (1.506, 27.863).. controls (1.489, 27.96) and (1.482, 28.061) .. (1.478, 28.159).. controls (1.473, 28.278) and (1.473, 28.397) .. (1.477, 28.515);
            \draw[knot_diagram] (1.478, 27.21).. controls (1.495, 27.416) and (1.533, 27.628) .. (1.651, 27.797).. controls (1.71, 27.882) and (1.789, 27.954) .. (1.883, 27.997).. controls (1.977, 28.04) and (2.086, 28.053) .. (2.185, 28.025).. controls (2.268, 28.002) and (2.343, 27.95) .. (2.394, 27.88).. controls (2.445, 27.811) and (2.473, 27.724) .. (2.471, 27.638).. controls (2.468, 27.552) and (2.437, 27.466) .. (2.383, 27.399).. controls (2.328, 27.332) and (2.251, 27.284) .. (2.167, 27.264).. controls (2.062, 27.239) and (1.948, 27.258) .. (1.853, 27.309).. controls (1.774, 27.351) and (1.707, 27.413) .. (1.655, 27.486);
            \draw[knot_diagram, ->] (1.505, 26.413).. controls (1.465, 26.676) and (1.456, 26.944) .. (1.478, 27.21);
        \end{tikzpicture}
        \hspace{0.7cm}
        \begin{tikzpicture}[scale=0.8, baseline={([yshift=-1ex]current bounding box.center)}]
            \begin{scope}[decoration={
                markings,
                mark=at position 0.75 with {\arrow{>}}}
                ]
                \draw[knot_diagram, postaction={decorate}] (-0.5, -1) .. controls (-0.15, -0.5) and (-0.15, 0.5) .. (-0.5, 1);
                \draw[knot_diagram, postaction={decorate}] (0.5, -1) .. controls (0.15, -0.5) and (0.15, 0.5) .. (0.5, 1);
            \end{scope}
        \end{tikzpicture}
        $\xrightarrow{\ R_2\ }$
        \begin{tikzpicture}[scale=0.75, baseline={([yshift=-1ex]current bounding box.center)}]
            \draw[knot_diagram] (3.675, 27.45).. controls (3.788, 27.388) and (3.912, 27.345) .. (4.039, 27.327);
            \draw[knot_diagram, <-] (3.121, 28.357).. controls (3.129, 28.107) and (3.227, 27.861) .. (3.393, 27.675).. controls (3.425, 27.639) and (3.459, 27.606) .. (3.496, 27.575);
            \draw[knot_diagram] (3.489, 29.2).. controls (3.446, 29.163) and (3.406, 29.122) .. (3.37, 29.079).. controls (3.203, 28.879) and (3.113, 28.617) .. (3.121, 28.357);
            \draw[knot_diagram] (4.035, 29.454).. controls (3.912, 29.435) and (3.793, 29.394) .. (3.683, 29.335);
            \draw[knot_diagram, <-] (4.039, 28.357).. controls (4.031, 28.107) and (3.933, 27.861) .. (3.767, 27.675).. controls (3.601, 27.489) and (3.368, 27.363) .. (3.121, 27.327);
            \draw[knot_diagram] (3.125, 29.454).. controls (3.383, 29.415) and (3.624, 29.279) .. (3.79, 29.079).. controls (3.957, 28.879) and (4.047, 28.617) .. (4.039, 28.357);
        \end{tikzpicture}
        \hspace{0.7cm}
        \begin{tikzpicture}[scale=0.75, baseline={([yshift=-1ex]current bounding box.center)}]
            \draw[knot_diagram] (5.443, 27.707) -- (4.55, 27.707);
            \draw[knot_diagram, ->] (6.973, 27.707) -- (5.443, 27.707);
            \draw[knot_diagram] (6.379, 27.632) -- (6.619, 27.216);
            \draw[knot_diagram] (6.101, 28.113) -- (6.289, 27.788);
            \draw[knot_diagram, ->] (5.428, 29.279) -- (6.101, 28.113);
            \draw[knot_diagram] (5.833, 28.791) -- (6.122, 29.267);
            \draw[knot_diagram] (5.588, 28.387) -- (5.706, 28.581);
            \draw[knot_diagram, ->] (5.247, 27.825) -- (5.588, 28.387);
            \draw[knot_diagram] (4.868, 27.201) -- (5.106, 27.593);
        \end{tikzpicture}
        $\xrightarrow{\ R_3\ }$
        \begin{tikzpicture}[scale=0.75, baseline={([yshift=-1ex]current bounding box.center)}]
            \begin{scope}[decoration={
                markings,
                mark=at position 0.65 with {\arrow{>}}}
                ]
                \draw[knot_diagram, postaction={decorate}] (9.369, 28.925) -- (7.107, 28.925);
                \draw[knot_diagram, postaction={decorate}] (8.289, 28.146).. controls (8.448, 28.379) and (8.612, 28.62) .. (8.752, 28.825);
            \end{scope}
            \begin{scope}[decoration={
                markings,
                mark=at position 0.35 with {\arrow{>}}}
                ]
                \draw[knot_diagram, postaction={decorate}] (7.628, 28.849).. controls (8.097, 28.209) and (8.858, 27.165) .. (8.858, 27.165);
            \end{scope}
            \draw[] (7.268, 29.337).. controls (7.272, 29.333) and (7.363, 29.21) .. (7.499, 29.024);
            \draw[] (8.881, 29.016).. controls (9.018, 29.217) and (9.109, 29.351) .. (9.109, 29.351);
            \draw (7.635, 27.21).. controls (7.648, 27.215) and (7.874, 27.541) .. (8.145, 27.936);
        \end{tikzpicture}
        \caption{\label{Figure:OrientedReidemeister}Reidemeister moves for oriented links}
    \end{figure}

    \emph{$R_1^r$ and $R_1^l$ moves}. Consider the case of the move $R_1^r$. Denote $a_1\in A(D_1)$ the left side and $a_2\in A(D_1)$ the right side of the diagram $D_1$. Denote the corresponding areas of $A(D_2)$ by $a_1'$ and $a_2'$, and let $\overline{a}\in A(D_2)$ be a new area (figure \ref{Figure:R1Move}).

    \begin{figure}[ht]
        \begin{tikzpicture}[scale=1.05]
            \begin{scope}[decoration={
                markings,
                mark=at position 0.75 with {\arrow{>}}}
                ]
                \draw[knot_diagram, postaction={decorate}] (0, -1) -- (0, 1);
            \end{scope}
            \draw (-0.5, 0) node {$a_1$};
            \draw (0.5, 0) node {$a_2$};
        \end{tikzpicture}
        \hspace{1.5cm}
        \begin{tikzpicture}[scale=0.9]
            \draw[knot_diagram] (1.477, 28.515).. controls (1.48, 28.615) and (1.487, 28.714) .. (1.496, 28.813);
            \draw[knot_diagram, ->] (1.506, 27.863).. controls (1.489, 27.96) and (1.482, 28.061) .. (1.478, 28.159).. controls (1.473, 28.278) and (1.473, 28.397) .. (1.477, 28.515);
            \draw[knot_diagram] (1.478, 27.21).. controls (1.495, 27.416) and (1.533, 27.628) .. (1.651, 27.797).. controls (1.71, 27.882) and (1.789, 27.954) .. (1.883, 27.997).. controls (1.977, 28.04) and (2.086, 28.053) .. (2.185, 28.025).. controls (2.268, 28.002) and (2.343, 27.95) .. (2.394, 27.88).. controls (2.445, 27.811) and (2.473, 27.724) .. (2.471, 27.638).. controls (2.468, 27.552) and (2.437, 27.466) .. (2.383, 27.399).. controls (2.328, 27.332) and (2.251, 27.284) .. (2.167, 27.264).. controls (2.062, 27.239) and (1.948, 27.258) .. (1.853, 27.309).. controls (1.774, 27.351) and (1.707, 27.413) .. (1.655, 27.486);
            \draw[knot_diagram, ->] (1.505, 26.413).. controls (1.465, 26.676) and (1.456, 26.944) .. (1.478, 27.21);

            \draw (1.0, 27.65) node {$a_1'$};
            \draw (3.0, 27.65) node {$a_2'$};
            \draw (2.05, 27.65) node {$\overline{a}$};
        \end{tikzpicture}
        \caption{\label{Figure:R1Move}Areas before (on the left) and after (on the right) move $R_1^r$}
    \end{figure}
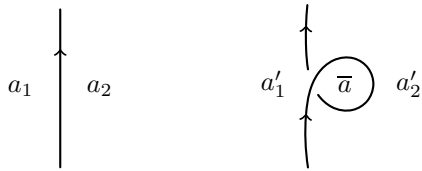

    Let $\mathcal{E}(D_1) = <\mathcal{G} | \mathcal{R}>$, where $\mathcal{G}$ is the set of generators and $\mathcal{R}$ is the set of relations. Let $x$ be a new generator assigned to the new crossing of the diagram $D_2$. Then $\mathcal{E}(D_2) = <\mathcal{G}, x | \mathcal{R}'>$, where $$\mathcal{R}' = (\mathcal{R}\setminus \{R_{a_1}, R_{a_2}\}) \cup \{R_{a_1'}, R_{a_2'}, R_{\overline{a}}\}.$$ Find that 
    \begin{align*}
        &R_{a_1'} = R_{a_1}\cdot a^{-1} x a, \\
        &R_{a_2'} = R_{a_2}\cdot b^{-1}a^{-1} x^{-1}\cdot x ab, \\
        &R_{\overline{a}} = b^{-1}x^{-1}b.
    \end{align*}

    It's clear that the generator $x$ can be excluded, so the groups $\mathcal{E}(D_1)$ and $\mathcal{E}(D_2)$ are isomorphic.

    For the case of the move $R_1^l$ the proof is similar.

    \emph{$R_2$ move}. Let $a_1, a_2, a_3\in A(D_1)$ denote the left, centre and right areas of the diagram $D_1$ respectively, and $a_1', a_2', a_2'', a_3'\in A(D_2)$ denote the corresponding areas of the diagram $D_2$ (Figure \ref{Figure:R2Move}). The area $a_2$ is split into two areas $a_2', a_2''$ and the new area $\overline{a}\in A(D_2)$ appears.

    \begin{figure}[ht]
        \begin{tikzpicture}[scale=1.1, baseline={([yshift=-10.5ex]current bounding box.center)}]
            \begin{scope}[decoration={
                markings,
                mark=at position 0.75 with {\arrow{>}}}
                ]
                \draw[knot_diagram, postaction={decorate}] (-0.5, -1) .. controls (-0.15, -0.5) and (-0.15, 0.5) .. (-0.5, 1);
                \draw[knot_diagram, postaction={decorate}] (0.5, -1) .. controls (0.15, -0.5) and (0.15, 0.5) .. (0.5, 1);
                \draw (-0.75, 0.0) node {$a_1$};
                \draw (0.0, 0.0) node {$a_2$};
                \draw (0.75, 0.0) node {$a_3$};
            \end{scope}
        \end{tikzpicture}
        \hspace{1.2cm}
        \begin{tikzpicture}[scale=1.0]
            \draw[knot_diagram] (3.675, 27.45).. controls (3.788, 27.388) and (3.912, 27.345) .. (4.039, 27.327);
            \draw[knot_diagram, <-] (3.121, 28.357).. controls (3.129, 28.107) and (3.227, 27.861) .. (3.393, 27.675).. controls (3.425, 27.639) and (3.459, 27.606) .. (3.496, 27.575);
            \draw[knot_diagram] (3.489, 29.2).. controls (3.446, 29.163) and (3.406, 29.122) .. (3.37, 29.079).. controls (3.203, 28.879) and (3.113, 28.617) .. (3.121, 28.357);
            \draw[knot_diagram] (4.035, 29.454).. controls (3.912, 29.435) and (3.793, 29.394) .. (3.683, 29.335);
            \draw[knot_diagram, <-] (4.039, 28.357).. controls (4.031, 28.107) and (3.933, 27.861) .. (3.767, 27.675).. controls (3.601, 27.489) and (3.368, 27.363) .. (3.121, 27.327);
            \draw[knot_diagram] (3.125, 29.454).. controls (3.383, 29.415) and (3.624, 29.279) .. (3.79, 29.079).. controls (3.957, 28.879) and (4.047, 28.617) .. (4.039, 28.357);

            \draw (3.6, 28.4) node {$\overline{a}$};
            \draw (2.6, 28.4) node {$a_1'$};
            \draw (4.6, 28.4) node {$a_3'$};
            \draw (3.6, 29.8) node {$a_2''$};
            \draw (3.6, 27.1) node {$a_2'$};
        \end{tikzpicture}
        \caption{\label{Figure:R2Move}Areas before (on the left) and after (on the right) move $R_2$}
    \end{figure}
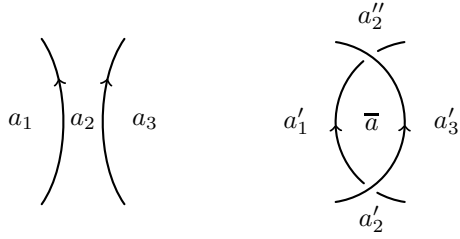

    Let $\mathcal{E}(D_1) = <\mathcal{G} | \mathcal{R}>$. Let $w_1, w_2$ be two words in the alphabet $\mathcal{G}$ such that $R_{a_2} = w_1\cdot w_2$, where $w_1$ corresponds to the lower part of $a_2$ and $w_2$ to the upper part.
    
    Let $x, y$ be new generators of the group $\mathcal{E}(D_2)$ ($x$ for the positive crossing and $y$ for the negative one). Then $\mathcal{E}(D_2) = <\mathcal{G}, x, y | \mathcal{R}'>$, where $$\mathcal{R}' = (\mathcal{R}\setminus \{R_{a_1}, R_{a_2}, R_{a_3}\})\cup \{R_{a_1'}, R_{a_2'}, R_{a_2''}, R_{\overline{a}}\}.$$

    Find that 
    \begin{align*}
        &R_{a_1'} = R_{a_1}\cdot a^{-1} x a \cdot a^{-1}y^{-1}a, \\
        &R_{a_3'} = R_{a_3}\cdot b^{-1} yb \cdot b^{-1}x^{-1}b, \\
        &R_{\overline{a}} = b^{-1}a^{-1}x^{-1}\cdot ya b, \\
        &R_{a_2'} = w_1\cdot xab, \\
        &R_{a_2''} = w_2\cdot b^{-1}a^{-1} y^{-1}.
    \end{align*}

    The relation $R_{\overline{a}}$ allows to exclude the generator $y$ (and replace everywhere $y$ by $x$). As a result, $R_{a_1'}$ and $R_{a_3'}$ will coincide with $R_{a_1}$ and $R_{a_3}$ respectively. Next, the relation $R_{a_2''} = w_2b^{-1}a^{-1}x^{-1}$ allows to exclude the generator $x$. After that the relation $R_{a_2'}$ will coincide with $R_{a_2}$.

    \emph{$R_3$ move}. Denote $c_1, c_2, c_3\in C(D_1)$ the crossings of the diagram $D_1$ used in the move, and denote $c_1', c_2', c_3'\in C(D_2)$ the corresponding crossings of the diagram $D_2$. Also denote $a_1, \ldots, a_7\in A(D_1)$ areas of the diagram $D_1$, and $a_1',\ldots, a_7'\in A(D_2)$ corresponding areas of the diagram $D_2$ (Figure \ref{Figure:R3Move}). Then 
    \begin{center}
        $\mathcal{E}(D_1) = <\mathcal{G}, x_1, x_2, x_3 | \mathcal{R}, R_{a_1}, \ldots, R_{a_7}>$,

        $\mathcal{E}(D_2) = <\mathcal{G}, y_1, y_2, y_3 | \mathcal{R}, R_{a_1'}, \ldots, R_{a_7'}>$,
    \end{center}
    where the generators $x_1, x_2, x_3$ correspond to the crossings $c_1, c_2, c_3$, and the generators $y_1, y_2, y_3$ correspond to the crossings $c_1', c_2', c_3'$. All other generators (the set $\mathcal{G}$) of the groups $\mathcal{E}(D_1)$ and $\mathcal{E}(D_2)$ are the same. The set of relations $\mathcal{R}$ is also the same for the groups $\mathcal{E}(D_1)$ and $\mathcal{E}(D_2)$, because it corresponds to areas which are not used in the move $R_3$.

    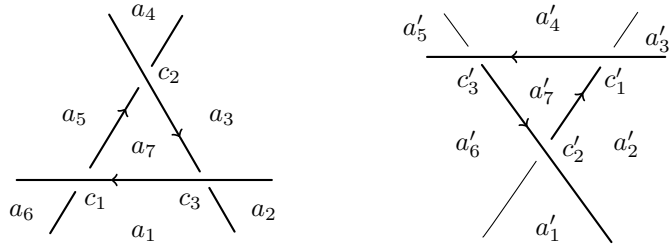
\begin{figure}[ht]
        \begin{tikzpicture}[scale=1.4]
            \draw[knot_diagram] (5.443, 27.707) -- (4.55, 27.707);
            \draw[knot_diagram, ->] (6.973, 27.707) -- (5.443, 27.707);
            \draw[knot_diagram] (6.379, 27.632) -- (6.619, 27.216);
            \draw[knot_diagram] (6.101, 28.113) -- (6.289, 27.788);
            \draw[knot_diagram, ->] (5.428, 29.279) -- (6.101, 28.113);
            \draw[knot_diagram] (5.833, 28.791) -- (6.122, 29.267);
            \draw[knot_diagram] (5.588, 28.387) -- (5.706, 28.581);
            \draw[knot_diagram, ->] (5.247, 27.825) -- (5.588, 28.387);
            \draw[knot_diagram] (4.868, 27.201) -- (5.106, 27.593);

            \draw (5.76, 27.2) node {$a_1$};
            \draw (6.9, 27.4) node {$a_2$};
            \draw (6.5, 28.3) node {$a_3$};
            \draw (5.76, 29.3) node {$a_4$};
            \draw (5.1, 28.3) node {$a_5$};
            \draw (4.6, 27.4) node {$a_6$};
            \draw (5.76, 28.0) node {$a_7$};

            \draw (5.3, 27.5) node {$c_1$};
            \draw (6.2, 27.5) node {$c_3$};
            \draw (6.0, 28.7) node {$c_2$};
        \end{tikzpicture}
        \hspace{1.2cm}
        \begin{tikzpicture}[scale=1.4]
            \begin{scope}[decoration={
                markings,
                mark=at position 0.65 with {\arrow{>}}}
                ]
                \draw[knot_diagram, postaction={decorate}] (9.369, 28.925) -- (7.107, 28.925);
                \draw[knot_diagram, postaction={decorate}] (8.289, 28.146).. controls (8.448, 28.379) and (8.612, 28.62) .. (8.752, 28.825);
            \end{scope}
            \begin{scope}[decoration={
                markings,
                mark=at position 0.35 with {\arrow{>}}}
                ]
                \draw[knot_diagram, postaction={decorate}] (7.628, 28.849).. controls (8.097, 28.209) and (8.858, 27.165) .. (8.858, 27.165);
            \end{scope}
            \draw[] (7.268, 29.337).. controls (7.272, 29.333) and (7.363, 29.21) .. (7.499, 29.024);
            \draw[] (8.881, 29.016).. controls (9.018, 29.217) and (9.109, 29.351) .. (9.109, 29.351);
            \draw (7.635, 27.21).. controls (7.648, 27.215) and (7.874, 27.541) .. (8.145, 27.936);

            \draw (8.26, 27.3) node {$a_1'$};
            \draw (9.0, 28.1) node {$a_2'$};
            \draw (9.3, 29.1) node {$a_3'$};
            \draw (8.26, 29.3) node {$a_4'$};
            \draw (7.0, 29.2) node {$a_5'$};
            \draw (7.5, 28.1) node {$a_6'$};
            \draw (8.2, 28.6) node {$a_7'$};

            \draw (8.9, 28.7) node {$c_1'$};
            \draw (7.5, 28.7) node {$c_3'$};
            \draw (8.5, 28.05) node {$c_2'$};
        \end{tikzpicture}
        \caption{\label{Figure:R3Move}Areas and crossings before (on the left) and after (on the right) move $R_3$}
    \end{figure}

    Find that
    \begin{center}
        $\begin{array}{ll}
            R_{a_1} = b^{-1} a^{-1} x_3^{-1}\cdot x_1 a b \cdot w_1, & R_{a_1'} = b^{-1} y_3^{-1} b\cdot w_1, \\
            R_{a_1} = a^{-1} x_3 a \cdot w_2, & R_{a_1'} = y_1 ab \cdot b^{-1} a^{-1} y_3^{-1} \cdot w_2, \\
            R_{a_3} = b^{-1} a^{-1} x_2^{-1}\cdot x_3 ab\cdot w_3, & R_{a_3'} = b^{-1} y_1 b\cdot w_3, \\
            R_{a_4} = a^{-1} x_2 a \cdot w_4, & R_{a_4'} = y_2ab\cdot b^{-1} a^{-1} y_1^{-1}\cdot w_4, \\
            R_{a_5} = b^{-1} a^{-1}x_1^{-1}\cdot x_2 ab \cdot w_5, & R_{a_5'} = b^{-1} y_2^{-1} b\cdot w_5, \\
            R_{a_6} = a^{-1} x_1^{-1} a \cdot w_6, & R_{a_6'} = y_3 ab\cdot b^{-1} a^{-1} y_2^{-1}\cdot w_6, \\
            R_{a_7} = b^{-1}x_1 b\cdot b^{-1} x_3^{-1} b \cdot b^{-1} x_2^{-1} b, & R_{a_7'} = a^{-1} y_3 a\cdot a^{-1} y_1^{-1} a\cdot a^{-1}y_2 a,
        \end{array}$
    \end{center}
    where $w_1, \ldots, w_6$ are fixed words in the alphabet $\mathcal{G}$.

    It's easy to see that the presentation of the group $\mathcal{E}(D_2)$ can be obtained from the presentation of the group $\mathcal{E}(D_1)$ by replacing 
    \begin{center}
        $\begin{array}{l}
            y_1 = a^{-1}x_2^{-1} x_3 a, \\
            y_2 = a^{-1} x_2^{-1} x_1 a, \\
            y_3 = a^{-1} x_1^{-1} x_3 a.
        \end{array}$
    \end{center}
    So these groups are isomorphic.
\end{proof}

For any oriented link $K$ we can define the group $\mathcal{E}(K)$ which coincides with the group $\mathcal{E}(D)$ for any diagram $D$ of $K$. This group $\mathcal{E}(K)$ is called the \emph{electric group} of the link $K$.

\begin{remark}
    It's possible to define \emph{reduced electric group} $\mathcal{E}_r(K)$ in the same way as $\mathcal{E}(K)$, but making the generators $a$ and $b$ trivial.
\end{remark}

\begin{example}
    The reduced electric group of the <<trefoil>> knot is isomorphic to $\mathbb{Z}_3.$
\end{example}

\begin{proposition}
    \label{Proposition:InverseOrientation}
    Let $K$ be an oriented knot, and let $K'$ be the same knot with inverse orientation. Then the groups $\mathcal{E}_r(K)$ and $\mathcal{E}_r(K')$ are isomorphic.
\end{proposition}
\begin{proof}
    Choose diagrams $D$ and $D'$ of knots $K$ and $K'$ which differ only in orientation. Let 
    \begin{center}
        $\mathcal{E}_r(D) = <x_1, \ldots, x_n | r_1, \ldots, r_k>$ and $\mathcal{E}_r(D) = <\xi_1, \ldots, \xi_n | \rho_1, \ldots, \rho_k>$.
    \end{center}

    The isomorphism between the groups $\mathcal{E}_r(D) $ and $\mathcal{E}_{r}(D')$ is defined by the map $x_i\mapsto \xi_i^{-1}$ (figure \ref{Figure:OrientationFlip}).

    \begin{figure}[ht]
        \begin{tikzpicture}[scale=0.75, baseline={([yshift=-2.0ex]current bounding box.center)}]
            \begin{scope}[decoration={
                markings,
                mark=at position 0.75 with {\arrow{>}}}
                ]
                \draw[knot_diagram, postaction={decorate}] (-1, -1) -- (1, 1);
                \draw[knot_diagram, postaction={decorate}] (1, -1) -- (0.15, -0.15);
                \draw[knot_diagram, postaction={decorate}] (-0.15, 0.15) -- (-1, 1);
            \end{scope}
    
            \draw (0, -0.75) node {$x_i$};
            \draw (1.0, 0) node {$x_i^{-1}$};
            \draw (0, 1.0) node {$x_i^{-1}$};
            \draw (-1.0, 0) node {$x_i$};
        \end{tikzpicture}
        $\mapsto$
        \begin{tikzpicture}[scale=0.75, baseline={([yshift=-2.0ex]current bounding box.center)}]
            \begin{scope}[decoration={
                markings,
                mark=at position 0.25 with {\arrow{<}}}
                ]
                \draw[knot_diagram, postaction={decorate}] (-1, -1) -- (1, 1);
                \draw[knot_diagram, postaction={decorate}] (1, -1) -- (0.15, -0.15);
                \draw[knot_diagram, postaction={decorate}] (-0.15, 0.15) -- (-1, 1);
            \end{scope}
    
            \draw (0, -0.75) node {$\xi_i^{-1}$};
            \draw (1.0, 0) node {$\xi_i$};
            \draw (0, 1.0) node {$\xi_i$};
            \draw (-1.0, 0) node {$\xi_i^{-1}$};
        \end{tikzpicture}
        \hspace{1.5cm}
        \begin{tikzpicture}[scale=0.75, baseline={([yshift=-2.0ex]current bounding box.center)}]
            \begin{scope}[decoration={
                markings,
                mark=at position 0.75 with {\arrow{>}}}
                ]
                \draw[knot_diagram, postaction={decorate}] (-1, -1) -- (-0.15, -0.15);
                \draw[knot_diagram, postaction={decorate}] (0.15, 0.15) -- (1, 1);
                \draw[knot_diagram, postaction={decorate}] (1, -1) -- (-1, 1);
            \end{scope}
    
            \draw (0, -0.75) node {$x_i$};
            \draw (1.0, 0) node {$x_i$};
            \draw (0, 1.0) node {$x_i^{-1}$};
            \draw (-1.0, 0) node {$x_i^{-1}$};
        \end{tikzpicture}
        $\mapsto$
        \begin{tikzpicture}[scale=0.75, baseline={([yshift=-2.0ex]current bounding box.center)}]
            \begin{scope}[decoration={
                markings,
                mark=at position 0.25 with {\arrow{<}}}
                ]
                \draw[knot_diagram, postaction={decorate}] (-1, -1) -- (-0.15, -0.15);
                \draw[knot_diagram, postaction={decorate}] (0.15, 0.15) -- (1, 1);
                \draw[knot_diagram, postaction={decorate}] (1, -1) -- (-1, 1);
            \end{scope}
    
            \draw (0, -0.75) node {$\xi_i^{-1}$};
            \draw (1.0, 0) node {$\xi_i^{-1}$};
            \draw (0, 1.0) node {$\xi_i$};
            \draw (-1.0, 0) node {$\xi_i$};
        \end{tikzpicture}
        \caption{\label{Figure:OrientationFlip}Isomorphism between groups $\mathcal{E}_{r}(D)$ and $\mathcal{E}_{r}(D')$}
    \end{figure}
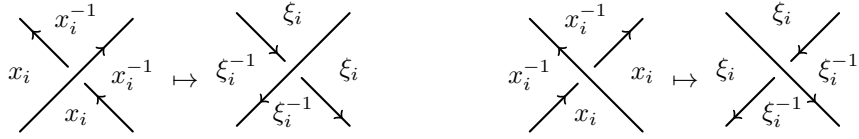
\end{proof}

\section{Group colourings}

The electric group is the group-valued links invariant. One way to study these groups is to consider homomorphisms from the electric group to finite groups.

Let $K$ be an oriented link and let $G$ be a finite group. The number of homomorphisms from the electric group $\mathcal{E}(K)$ to the group $G$ is finite. This number is an integer invariant of the link $K$. Each homomorphism $\xi\colon \mathcal{E}(K)\to G$ can be described in terms of colourings of the diagram of the link $K$.

Fix two group elements $\alpha, \beta\in G$. Let $D$ be a diagram of the link $K$ on an oriented 2-sphere $S^2$. As before, denote $C(D) = \{c_1, \ldots, c_n\}$ the set of crossings of the diagram $D$, and $A(D) = \{a_1, \ldots, a_{n + 2}\}$ the set of complement areas. Let $\xi\colon C(D)\to G$ be a map which assigns a group element to each crossing of the diagram $D$. We will refer to the map $\xi$ as the \emph{colouring} of the diagram $D$.

Let $a_i\in A(D)$ be an area of the diagram $D$, and let $p_i$ be a point on the boundary of $a_i$ that does not coincide with any crossing. Then construct an element $g_{a_i, p_i}\in G$ as follows. At the beginning $g_{a_i, p_i} = e$ --- the neutral element of the group $G$. Then walk along the boundary of $a_i$ starting from $p_i$ and when meeting the positive crossing $c_j$, multiply $g_{a_i, p_i}$ by one of the following elements $\xi(c_j)\alpha \beta$, $\beta^{-1}\alpha^{-1}\xi(c_j)^{-1}$, $\alpha^{-1}\xi(c_j)\alpha$ or $\beta^{-1}\xi(c_j)^{-1}\beta$. Which element to choose depends on which corner the area $a_i$ belongs to in the neighbourhood of the crossing $c_j$ (figure \ref{Figure:GroupCorners} on the left). If the crossing $c_j$ is negative, then multiply $g_{a_i, p_i}$ by one of the following elements $\xi(c_j)\alpha \beta$, $\beta^{-1}\alpha^{-1}\xi(c_j)^{-1}$, $\alpha^{-1}\xi(c_j)^{-1}\alpha$ or $\beta^{-1}\xi(c_j)\beta$ (figure \ref{Figure:GroupCorners} on the right).

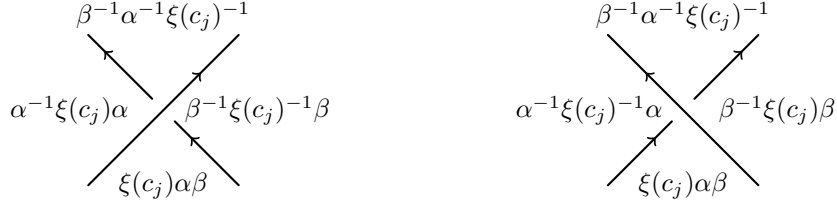
\begin{figure}[ht]
    \begin{tikzpicture}
        \begin{scope}[decoration={
            markings,
            mark=at position 0.75 with {\arrow{>}}}
            ]
            \draw[knot_diagram, postaction={decorate}] (-1, -1) -- (1, 1);
            \draw[knot_diagram, postaction={decorate}] (1, -1) -- (0.15, -0.15);
            \draw[knot_diagram, postaction={decorate}] (-0.15, 0.15) -- (-1, 1);
        \end{scope}

        \draw (0, -1.0) node {$\xi(c_j) \alpha \beta$};
        \draw (1.25, 0) node {$\beta^{-1} \xi(c_j)^{-1} \beta$};
        \draw (0, 1.25) node {$\beta^{-1} \alpha^{-1} \xi(c_j)^{-1}$};
        \draw (-1.25, 0) node {$\alpha^{-1} \xi(c_j) \alpha$};
    \end{tikzpicture}
    \hspace{2cm}
    \begin{tikzpicture}
        \begin{scope}[decoration={
            markings,
            mark=at position 0.75 with {\arrow{>}}}
            ]
            \draw[knot_diagram, postaction={decorate}] (-1, -1) -- (-0.15, -0.15);
            \draw[knot_diagram, postaction={decorate}] (0.15, 0.15) -- (1, 1);
            \draw[knot_diagram, postaction={decorate}] (1, -1) -- (-1, 1);
        \end{scope}

        \draw (0, -1.0) node {$\xi(c_j) \alpha \beta$};
        \draw (1.25, 0) node {$\beta^{-1} \xi(c_j) \beta$};
        \draw (0, 1.25) node {$\beta^{-1} \alpha^{-1} \xi(c_j)^{-1}$};
        \draw (-1.25, 0) node {$\alpha^{-1} \xi(c_j)^{-1} \alpha$};
    \end{tikzpicture}
    \caption{\label{Figure:GroupCorners}Multipliers of the element $g_{a_i, p_i}$ corresponding to the crossing $c_j$ (positive on the left, negative on the right)}
\end{figure}

We will say that the colouring $\xi$ is \emph{proper $(\alpha, \beta)$-colouring} if for each $a_i\in A(D)$: $g_{a_i, p_i} = e$. It's clear that the definition of the proper colouring does not depend on the choice of the points $p_i$ in the boundary of the area $a_i$. Denote the set of all proper $(\alpha, \beta)$-colourings of the diagram $D$ by $Col_{\alpha, \beta}(D)$.

\begin{remark}
    It's clear that $$|Hom(\mathcal{E}(D)\to G)| = \sum\limits_{\alpha, \beta\in G}|Col_{\alpha, \beta}(D)|.$$
\end{remark}

\begin{theorem}
    \label{Theorem:ColouringsBijection}
    Let $G$ be a finite group, $\alpha, \beta\in G$. Let $D_1, D_2$ be two diagrams of the oriented link $K$. Then there is a natural bijection between sets of proper $(\alpha, \beta)$-colourings $Col_{\alpha, \beta}(D_1)$ and $Col_{\alpha, \beta}(D_2)$.
\end{theorem}
\begin{proof}
    The main idea of the proof is similar to the one used in the proof of the theorem \ref{Theorem:ElectricGroup}. We should prove that for any proper $(\alpha, \beta)$-colouring $\xi\in Col_{\alpha, \beta}(D_1)$ of the diagram $D_1$ there is exactly one proper $(\alpha, \beta)$-colouring $\xi'\in Col_{\alpha, \beta}(D_2)$ of the diagram $D_2$, where $D_2$ is obtained from $D_1$ by one of the Reidemeister moves $R_1^l,R_1^r, R_2$ or $R_3$ (Figure \ref{Figure:OrientedReidemeister}). We will use the same notations as in the proof of the theorem \ref{Theorem:ElectricGroup}.

    \emph{$R_1^l$ and $R_1^{r}$ moves}. Let $c\in C(D_2)$ be a crossing that appears after the $R_1^l$ (or $R_1^r$). If $\xi'$ is the proper $(\alpha, \beta)$-colouring of the diagram $D_2$, then $b^{-1}\cdot \xi'(c)^{\pm 1}\cdot b = e$. So $\xi'(c) = e$. This colouring $\xi'$ is a unique colouring which extends the colouring $\xi$ of the diagram $D_1$.

    \emph{$R_2$ move}. Let $c_1, c_2\in C(D_2)$ be new crossings that appear after the move ($c_1$ is positive, $c_2$ is negative). Then $\beta^{-1}\alpha^{-1} \xi'(c_1)^{-1}\cdot \xi'(c_2) \alpha \beta = e$. So $\xi'(c_1) = \xi'(c_2)$. Then $g_{a_1', p_1'} = e$ and $g_{a_2', p_2'} = e$, where $p_i'$ is a starting point on the boundary of the are $a_i'$, $i = 1, 2$.

    Let $p$ be a starting point at the boundary of the area $a_2\in A(D_1)$ near the left arc used in the move. Then $g_{a_2, p} = g_1\cdot g_2 = e$, where the element $g_1\in G$ corresponds to the lower part of the area $a_2$ and $g_2\in G$ corresponds to the upper part. Then there is exactly one value $\xi'(c_1) = \xi'(c_2)$ such that 
    \begin{center}
        $g_1\cdot \xi'(c_1)\alpha \beta = e$ and $g_2 \cdot\beta^{-1}\alpha^{-1}\xi'(c_2)^{-1} = e$.
    \end{center}
    Find that $$\xi'(c_1) = \xi'(c_2) = g_2 \beta^{-1}\alpha^{-1} = g_1^{-1} \beta^{-1} \alpha^{-1}.$$

    \emph{$R_3$ move}. Let $\xi\in Col_{\alpha, \beta}(D_1)$ be a proper $(\alpha, \beta)$-colouring of the diagram $D_1$. Let $\xi(c_i) = s_i$, $i = 1, 2, 3$. Consider the colouring $\xi'$ of the diagram $D_2$, which coincides with $\xi$ at all crossings outside the move area, and $\xi'(c_1') = t_i$, $i = 1, 2, 3$.

    Using calculations within the proof of the theorem \ref{Theorem:ElectricGroup}, find that there is exactly one set of values 
    \begin{center}
        $t_1 = \alpha^{-1}s_2^{-1} s_3\alpha$, $t_2 = \alpha^{-1}s_2^{-1}s_1\alpha$ and $t_3 = \alpha^{-1}s_1^{-1}s_3 \alpha$,
    \end{center}
    such that the colouring $\xi'$ is the proper $(\alpha, \beta)$-colouring of the diagram $D_2$.
\end{proof}

\section{Tensor network invariants}

The main subject of this section is the construction of an invariant for links coloured by a finite group, similar to that of \cite{TYB}. In fact, our construction defines the functor from the category of coloured oriented tangles to the category of modules and linear maps (see \cite{TRM} for the theory of uncoloured oriented tangles).

\subsection{Tensors}

Let $K$ be a ring, and let $V$ be a finitely generated free module over the ring $K$ of rank $N\geqslant 1$. Any linear map $T\colon V^{\otimes p}\to V^{\otimes q}$ can naturally be interpreted as a tensor $\widetilde{T}\colon V^{\otimes p}\otimes (V^{*})^{\otimes q} \to K$ as follows. Fix a basis $\{e_1, \ldots, e_N\}$ of the module $V$ and let $\{e^1, \ldots, e^N\}$ be a dual basis of the dual module $V^{*}$. If the linear map $T$ is defined on the basis by the equation $$T(e_{i_1}\otimes \ldots\otimes e_{i_p}) = \sum\limits_{j_1, \ldots, j_q = 1}^{N}T_{i_1 \ldots i_p}^{j_1 \ldots j_q}e_{j_1}\otimes\ldots \otimes e_{j_q}, $$ then $$\widetilde{T}(e_{i_1}\otimes \ldots\otimes e_{i_p}\otimes e^{j_1}\otimes \ldots\otimes e^{j_q}) = T_{i_1 \ldots i_p}^{j_1 \ldots j_q}$$ for all $i_1, \ldots, i_p, j_1, \ldots, j_q\in\{1, \ldots, N\}$. So we can identify the linear map $T$ with the tensor $\widetilde{T}$ and denote it by the same letter $T$. The values $T_{i_1 \ldots i_p}^{j_1 \ldots j_q}$ are coordinates of the tensor $T$.

\subsubsection{Contractions}

Let $T\colon V^{\otimes p}\otimes (V^{*})^{\otimes q}\to K$ be a tensor, $p, q \geqslant 1$. Then the result of the contraction of $T$ along $s$-th lower index and $r$-th upper index is a tensor $*_{s}^{r}T\colon V^{\otimes p - 1}\otimes (V^{*})^{\otimes q - 1}\to K$ with coordinates $$(*_{s}^{r}T)^{j_1\ldots j_{q - 1}}_{i_1\ldots i_{p - 1}} = \sum\limits_{x = 1}^{N}T^{j_1\ldots j_{r - 1}x j_{r}\ldots j_{q - 1}}_{i_1\ldots i_{s - 1}x i_{s} \ldots i_{p - 1}}.$$ Similarly, we can define the operation of simultaneous contraction along several pairs $(s_1, r_1), \ldots, (s_k, r_k)$ of lower and upper indices. The result is denoted by $*_{s_1}^{r_1} *_{s_2}^{r_2}\ldots *_{s_k}^{r_k}T\colon V^{\otimes p - k}\otimes (V^{*})^{q - k}\to K$ and defined by the coordinates $$(*_{s_1}^{r_1} *_{s_2}^{r_2}\ldots *_{s_k}^{r_k}T)_{i_1\ldots i_{p - k}}^{j_1\ldots j_{q - k}} = \sum\limits_{x_1, \ldots, x_k = 1}^{N}T_{i_1\ldots x_1 \ldots x_2\ldots x_k\ldots i_{p - k}}^{j_1\ldots x_1 \ldots x_2\ldots x_k\ldots j_{q - k}},$$ where each index $x_l$, $l = 1, \ldots, k$, is placed at position $s_l$ at the bottom and $r_l$ at the top.

\subsubsection{Graphical approach}

It's convenient to use a graphical approach similar to that of \cite{KU}. It allows to construct new tensors by doing tensor products and contractions in a visual way. In these \emph{tensor diagrams} each tensor $T\colon V^{\otimes p}\otimes (V^{*})^{\otimes q}\to K$ is drawn as a square block with $p$ input arrows and $q$ output arrows (figure \ref{Figure:TensorDiagram} on the left). Input arrows are ordered counter-clockwise and identified with instances of the module $V$ on the argument of the tensor $T$. Output arrows ordered in clockwise direction and identified with instances of the dual module $V^{*}$ on the argument of $T$.

\begin{figure}[ht]
    \begin{tikzpicture}[baseline={([yshift=-1ex]current bounding box.center)}]
        \node at (0, 0) [square, draw] (t) {$T$};
        \begin{scope}[decoration={
            markings,
            mark=at position 0.75 with {\arrow{>}}}
            ]
            \draw[tensor_diagram, postaction={decorate}] (-0.75, -0.75) -- (t);
            \draw[tensor_diagram, postaction={decorate}] (0.75, -0.75) -- (t);
            \draw[tensor_diagram, postaction={decorate}] (t) -- (-0.75, 0.75);
            \draw[tensor_diagram, postaction={decorate}] (t) -- (0.75, 0.75);
        \end{scope}
        \draw (0, -0.75) node {$\ldots$};
        \draw (0, 0.75) node {$\ldots$};
    \end{tikzpicture}
    \hspace{1.5cm}
    \begin{tikzpicture}[baseline={([yshift=-1ex]current bounding box.center)}]
        \node at (0, 0) [square, draw] (t) {$T$};
        \node at (1, 0) [square, draw] (s) {$S$};
        \begin{scope}[decoration={
            markings,
            mark=at position 0.75 with {\arrow{>}}}
            ]
            \draw[tensor_diagram, postaction={decorate}] (-0.75, -0.75) -- (t);
            \draw[tensor_diagram, postaction={decorate}] (1.75, -0.75) -- (s);
            \draw[tensor_diagram, postaction={decorate}] (t) -- (-0.75, 0.75);
            \draw[tensor_diagram, postaction={decorate}] (s) -- (1.75, 0.75);
        \end{scope}
        \draw (0.5, -0.75) node {\makebox[1.5cm]{\dotfill}};
        \draw (0.5, 0.75) node {\makebox[1.5cm]{\dotfill}};
    \end{tikzpicture}
    \hspace{1.5cm}
    \begin{tikzpicture}[baseline={([yshift=-1ex]current bounding box.center)}]
        \node at (0, 0) [square, draw] (t) {$T$};
        \begin{scope}[decoration={
            markings,
            mark=at position 0.75 with {\arrow{>}}}
            ]
            \draw[tensor_diagram, postaction={decorate}] (-0.65, -0.65) -- (t);
            \draw[tensor_diagram, postaction={decorate}] (0.65, -0.65) -- (t);
            \draw[tensor_diagram, postaction={decorate}] (t) -- (-0.65, 0.65);
            \draw[tensor_diagram, postaction={decorate}] (t) -- (0.65, 0.65);
            \draw[tensor_diagram, postaction={decorate}] (0, 0.36) arc (180:0:0.75) -- (1.5, -0.36) arc(0:-180:0.75);
        \end{scope}
    \end{tikzpicture}
    \caption{\label{Figure:TensorDiagram}Tensor $T$ (on the left), tensor product $T\otimes S$ (on the center) and result of contraction $*_{s}^{r}(T)$ (on the right)}
\end{figure}
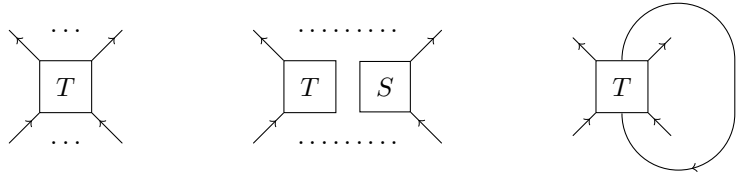

The tensor product of two tensors $T, S$ is just two blocks drawn together (figure \ref{Figure:TensorDiagram} on the center).

To make the contraction of the tensor $T$ along $s$-th lower index and $r$-th upper index, we simply connect the $r$-th output arrow with the $s$-th input arrow on the tensor diagram (figure \ref{Figure:TensorDiagram} on the right). Similarly, the result of simultaneous contraction is obtained by connecting several output arrows with corresponding input arrows.

\subsubsection{Invertible tensors}

The tensor $T\colon V^{\otimes p}\otimes (V^{*})^{\otimes p}\to K$ is called \emph{identity} if its coordinates are $$T_{i_1 \ldots i_p}^{j_1\ldots j_q} = \delta_{i_1}^{j_1}\cdot \ldots\cdot \delta_{i_p}^{j_p},$$ where $\delta_{i}^{j}$ is the Kronecker symbol ($\delta_{i}^{j} = 1$ if $i = j$, and $\delta_{i}^{j} = 0$ if $i \neq j$). It's clear that the identity tensor is just the identity map $id\colon V^{\otimes p}\to V^{\otimes p}$.

The tensor $T\colon V^{\otimes p}\otimes (V^{*})^{\otimes p}\to K$ is called \emph{invertible}, if there exists a tensor $\overline{T}\colon V^{\otimes p}\otimes (V^{*})^{\otimes p}\to K$ such that $$*^{p + 1}_{1}\ldots *^{2p}_{p}(T\otimes \overline{T}) = *^{p + 1}_{1}\ldots *^{2p}_{p}(\overline{T}\otimes T) = id,$$ where $id\colon V^{\otimes p}\otimes (V^{*})^{\otimes p}\to K$ is the identity tensor. The tensor $\overline{T}$ is called the \emph{inverse} of the tensor $T$. It's clear that if we consider the tensor $T$ as a map $T\colon V^{\otimes p}\to V^{\otimes p}$, then the inverse tensor $\overline{T}$ is just the inverse map $T^{-1}$.

\subsection{Consistent tensor $G$-systems}

Let $G$ be a finite group, $\alpha, \beta\in G$ be fixed elements. Consider three families of invertible tensors:
\begin{enumerate}
    \item $\mathcal{Y} = \{Y_g\colon V\otimes V\otimes V^{*}\otimes V^{*}\to K | g\in G\}$;
    \item $\mathcal{L} = \{L_{g_1, g_2}\colon V\otimes V^{*}\to K | g_1, g_2\in G\}$;
    \item $\mathcal{R} = \{R_{g_1, g_2}\colon V\otimes V^{*}\to K | g_1, g_2\in G\}$.
\end{enumerate}

For each $x\in G$ denote $\Gamma_x = *_{1}^{2}(R_{x, x}\otimes L_{x, x})$ and $\overline{\Gamma_x} = *_{1}^{2}(\overline{L_{x, x}}\otimes \overline{R_{x, x}})$. Diagrams of the tensors $\Gamma_x$ and $\overline{\Gamma_x}$ are shown in the figure \ref{Figure:GammaDiagrams}. It's clear that $\overline{\Gamma_x}$ is the inverse of $\Gamma_x$.

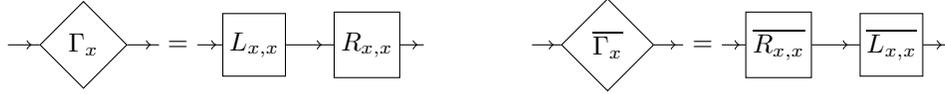
\begin{figure}[ht]
    \begin{tikzpicture}[baseline={([yshift=-0.5ex]current bounding box.center)}]
        \node (0, 0) [diamond, draw] (g) {$\Gamma_x$};
        \begin{scope}[decoration={
            markings,
            mark=at position 0.75 with {\arrow{>}}}
            ]
            \draw[tensor_diagram, postaction={decorate}] (-1.0, 0.0) -- (g);
            \draw[tensor_diagram, postaction={decorate}] (g) -- (1.0, 0.0);
        \end{scope}
    \end{tikzpicture}
    $=$
    \begin{tikzpicture}[baseline={([yshift=-0.5ex]current bounding box.center)}]
        \node at (0, 0) [square, draw, inner sep=-1pt] (r) {$L_{x, x}$};
        \node at (1.5, 0) [square, draw, inner sep=-1pt] (l) {$R_{x, x}$};
        \begin{scope}[decoration={
            markings,
            mark=at position 0.75 with {\arrow{>}}}
            ]
            \draw[tensor_diagram, postaction={decorate}] (r) -- (l);
            \draw[tensor_diagram, postaction={decorate}] (-0.75, 0.0) -- (r);
            \draw[tensor_diagram, postaction={decorate}] (l) -- (2.25, 0.0);
        \end{scope}
    \end{tikzpicture}
    \hspace{1.2cm}
    \begin{tikzpicture}[baseline={([yshift=-0.5ex]current bounding box.center)}]
        \node (0, 0) [diamond, draw] (g) {$\overline{\Gamma_x}$};
        \begin{scope}[decoration={
            markings,
            mark=at position 0.75 with {\arrow{>}}}
            ]
            \draw[tensor_diagram, postaction={decorate}] (-1.0, 0.0) -- (g);
            \draw[tensor_diagram, postaction={decorate}] (g) -- (1.0, 0.0);
        \end{scope}
    \end{tikzpicture}
    $=$
    \begin{tikzpicture}[baseline={([yshift=-0.5ex]current bounding box.center)}]
        \node at (0, 0) [square, draw, inner sep=-1pt] (r) {$\overline{R_{x, x}}$};
        \node at (1.5, 0) [square, draw, inner sep=-1pt] (l) {$\overline{L_{x, x}}$};
        \begin{scope}[decoration={
            markings,
            mark=at position 0.75 with {\arrow{>}}}
            ]
            \draw[tensor_diagram, postaction={decorate}] (r) -- (l);
            \draw[tensor_diagram, postaction={decorate}] (-0.75, 0.0) -- (r);
            \draw[tensor_diagram, postaction={decorate}] (l) -- (2.25, 0.0);
        \end{scope}
    \end{tikzpicture}
    \caption{\label{Figure:GammaDiagrams}Tensors $\Gamma_x$ and $\overline{\Gamma_x}$}
\end{figure}

Let $A, B\in K$ be invertible ring elements. We call the tuple $\mathcal{S} = (\mathcal{Y}, \mathcal{L}, \mathcal{R}, A, B)$ a \emph{consistent tensor $G$-system} if it satisfies the following conditions:
\begin{enumerate}
    \item For neutral element $e\in G$:
    \begin{enumerate}
        \item $*_{3}^{2} *_{2}^{3}(Y_{e}\otimes \overline{\Gamma_e}) = A B^{-1} \cdot id$;
        \item $*_{3}^{2} *_{2}^{3}(\overline{Y_{e}}\otimes \overline{\Gamma_e}) = A B \cdot id$;
        \item $*_{3}^{1} *_{1}^{3}(Y_{e}\otimes \Gamma_e) = A^{-1} B^{-1} \cdot id$;
        \item $*_{3}^{1} *_{1}^{3}(\overline{Y_{e}}\otimes \Gamma_e) = A^{-1} B \cdot id$.
    \end{enumerate}
    \item For all $x, y, z\in G$: $*_{3}^{1} *_{5}^{2} *_{4}^{5} (Y_x\otimes Y_z\otimes Y_y) = *_{5}^{2} *_{2}^{3} *_{6}^{4} (Y_{b}\otimes Y_{a}\otimes Y_{c})$, where $a = y \alpha x^{-1}\alpha ^{-1}$, $b = xz$ and $c = y \alpha z \alpha^{-1}$;
    \item For all $x, y_1, y_2\in G$:
    \begin{enumerate}
        \item $*_{3}^{1} *_{7}^{2} *_{5}^{3} *_{1}^{5} *_{4}^{6} (Y_x\otimes \overline{Y_x}\otimes \Gamma_x\otimes \overline{L_{y_1, x}}\otimes \overline{R_{x, y_2}}) = id$;
        \item $*_{7}^{1} *_{4}^{2} *_{5}^{4} *_{2}^{5} *_{3}^{6} (\overline{Y_x}\otimes Y_x\otimes \overline{\Gamma_x}\otimes R_{y_1, x}\otimes L_{x, y_2}) = id$;
    \end{enumerate}
    \item For all $x, y_1, y_2, y_3, y_4\in G$:
    \begin{enumerate}
        \item $*_{3}^{1} *_{4}^{2} *_{5}^{3} *_{6}^{4} (\overline{L_{y_4, x}}\otimes \overline{L_{y_3, x}}\otimes Y_{x}\otimes L_{x, y_2}\otimes L_{x y_1}) = *_{3}^{1} *_{4}^{2} *_{5}^{3} *_{6}^{4} (R_{y_4, x}\otimes R_{y_3, x}\otimes Y_{x}\otimes \overline{R_{x, y_2}}\otimes \overline{R_{x, y_1}})$;
        \item $*_{3}^{1} *_{4}^{2} *_{5}^{3} *_{6}^{4} (\overline{L_{y_4, x}}\otimes \overline{L_{y_3, x}}\otimes \overline{Y_{x}}\otimes L_{x, y_2}\otimes L_{x y_1}) = *_{3}^{1} *_{4}^{2} *_{5}^{3} *_{6}^{4} (R_{y_4, x}\otimes R_{y_3, x}\otimes \overline{Y_{x}}\otimes \overline{R_{x, y_2}}\otimes \overline{R_{x, y_1}})$.
    \end{enumerate}
\end{enumerate}

Axioms for consistent tensor $G$-systems look quite cumbersome, but they are natural in the sense of link invariants. Diagrams of axioms 1a -- 1d are shown on the figure \ref{Figure:Axioms1}, diagram of axiom 2 is shown on the figure \ref{Figure:Axioms2}, diagrams of axioms 3a -- 3b are shown on the figure \ref{Figure:Axioms3} and diagrams of axioms 4a -- 4b are shown on the figure \ref{Figure:Axioms4}.

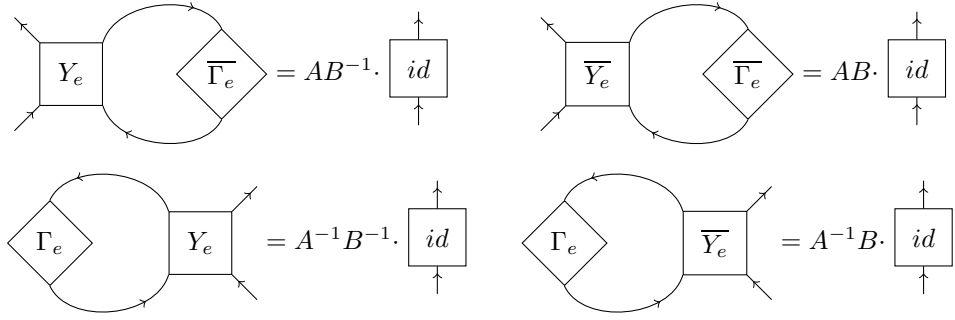
\begin{figure}[ht]
    \begin{tikzpicture}[baseline={([yshift=-0.5ex]current bounding box.center)}]
        \node at (0.0, 0.0) [square, draw] (y) {$Y_{e}$};
        \node at (2.0, 0.0) [diamond, draw] (g) {$\overline{\Gamma_e}$};
        \begin{scope}[decoration={
            markings,
            mark=at position 0.75 with {\arrow{>}}}
            ]
            \draw[tensor_diagram, postaction={decorate}] (y.north east) .. controls (0.6, 1.0) and (1.75, 1.1) .. (g.north);
            \draw[tensor_diagram, postaction={decorate}] (g.south) .. controls (1.75, -1.1) and (0.6, -1.0) .. (y.south east);
            \draw[tensor_diagram, postaction={decorate}] (y) -- (-0.75, 0.75);
            \draw[tensor_diagram, postaction={decorate}] (-0.75, -0.75) -- (y);
        \end{scope}
    \end{tikzpicture}
    $=AB^{-1}\cdot$
    \begin{tikzpicture}[baseline={([yshift=-1.0ex]current bounding box.center)}]
        \node at (0, 0) [square, draw] (s) {$id$};
        \begin{scope}[decoration={
            markings,
            mark=at position 0.75 with {\arrow{>}}}
            ]
            \draw[tensor_diagram, postaction={decorate}] (0.0, -0.75) -- (s);
            \draw[tensor_diagram, postaction={decorate}] (s) -- (0.0, 0.75);
        \end{scope}
    \end{tikzpicture}
    \hspace{1cm}
    \begin{tikzpicture}[baseline={([yshift=-0.5ex]current bounding box.center)}]
        \node at (0.0, 0.0) [square, draw] (y) {$\overline{Y_{e}}$};
        \node at (2.0, 0.0) [diamond, draw] (g) {$\overline{\Gamma_e}$};
        \begin{scope}[decoration={
            markings,
            mark=at position 0.75 with {\arrow{>}}}
            ]
            \draw[tensor_diagram, postaction={decorate}] (y.north east) .. controls (0.6, 1.0) and (1.75, 1.1) .. (g.north);
            \draw[tensor_diagram, postaction={decorate}] (g.south) .. controls (1.75, -1.1) and (0.6, -1.0) .. (y.south east);
            \draw[tensor_diagram, postaction={decorate}] (y) -- (-0.75, 0.75);
            \draw[tensor_diagram, postaction={decorate}] (-0.75, -0.75) -- (y);
        \end{scope}
    \end{tikzpicture}
    $=AB\cdot$
    \begin{tikzpicture}[baseline={([yshift=-1.0ex]current bounding box.center)}]
        \node at (0, 0) [square, draw] (s) {$id$};
        \begin{scope}[decoration={
            markings,
            mark=at position 0.75 with {\arrow{>}}}
            ]
            \draw[tensor_diagram, postaction={decorate}] (0.0, -0.75) -- (s);
            \draw[tensor_diagram, postaction={decorate}] (s) -- (0.0, 0.75);
        \end{scope}
    \end{tikzpicture}

    \begin{tikzpicture}[baseline={([yshift=-0.5ex]current bounding box.center)}]
        \node at (0.0, 0.0) [square, draw] (y) {$Y_{e}$};
        \node at (-2.0, 0.0) [diamond, draw] (g) {$\Gamma_e$};
        \begin{scope}[decoration={
            markings,
            mark=at position 0.75 with {\arrow{>}}}
            ]
            \draw[tensor_diagram, postaction={decorate}] (y.north west) .. controls (-0.6, 1.0) and (-1.75, 1.1) .. (g.north);
            \draw[tensor_diagram, postaction={decorate}] (g.south) .. controls (-1.75, -1.1) and (-0.6, -1.0) .. (y.south west);
            \draw[tensor_diagram, postaction={decorate}] (y) -- (0.75, 0.75);
            \draw[tensor_diagram, postaction={decorate}] (0.75, -0.75) -- (y);
        \end{scope}
    \end{tikzpicture}
    $=A^{-1}B^{-1}\cdot$
    \begin{tikzpicture}[baseline={([yshift=-1.0ex]current bounding box.center)}]
        \node at (0, 0) [square, draw] (s) {$id$};
        \begin{scope}[decoration={
            markings,
            mark=at position 0.75 with {\arrow{>}}}
            ]
            \draw[tensor_diagram, postaction={decorate}] (0.0, -0.75) -- (s);
            \draw[tensor_diagram, postaction={decorate}] (s) -- (0.0, 0.75);
        \end{scope}
    \end{tikzpicture}
    \hspace{0.5cm}
    \begin{tikzpicture}[baseline={([yshift=-0.5ex]current bounding box.center)}]
        \node at (0.0, 0.0) [square, draw] (y) {$\overline{Y_{e}}$};
        \node at (-2.0, 0.0) [diamond, draw] (g) {$\Gamma_e$};
        \begin{scope}[decoration={
            markings,
            mark=at position 0.75 with {\arrow{>}}}
            ]
            \draw[tensor_diagram, postaction={decorate}] (y.north west) .. controls (-0.6, 1.0) and (-1.75, 1.1) .. (g.north);
            \draw[tensor_diagram, postaction={decorate}] (g.south) .. controls (-1.75, -1.1) and (-0.6, -1.0) .. (y.south west);
            \draw[tensor_diagram, postaction={decorate}] (y) -- (0.75, 0.75);
            \draw[tensor_diagram, postaction={decorate}] (0.75, -0.75) -- (y);
        \end{scope}
    \end{tikzpicture}
    $=A^{-1}B\cdot$
    \begin{tikzpicture}[baseline={([yshift=-1.0ex]current bounding box.center)}]
        \node at (0, 0) [square, draw] (s) {$id$};
        \begin{scope}[decoration={
            markings,
            mark=at position 0.75 with {\arrow{>}}}
            ]
            \draw[tensor_diagram, postaction={decorate}] (0.0, -0.75) -- (s);
            \draw[tensor_diagram, postaction={decorate}] (s) -- (0.0, 0.75);
        \end{scope}
    \end{tikzpicture}
    \caption{\label{Figure:Axioms1}Tensor diagrams for the axioms 1a -- 1d}
\end{figure}

\begin{figure}[ht]
    \begin{tikzpicture}[baseline={([yshift=-1.0ex]current bounding box.center)}]
        \node at (0.0, 0.0) [square, draw] (x) {$Y_x$};
        \node at (1.25, 1.25) [square, draw] (y) {$Y_y$};
        \node at (0.0, 2.5) [square, draw] (z) {$Y_z$};
        \begin{scope}[decoration={
            markings,
            mark=at position 0.75 with {\arrow{>}}}
            ]
            \draw[tensor_diagram, postaction={decorate}] (-0.75, -0.75) -- (x.south west);
            \draw[tensor_diagram, postaction={decorate}] (0.75, -0.75) -- (x.south east);
            \draw[tensor_diagram, postaction={decorate}] (x.north west) .. controls (-0.75, 0.75) and (-0.75, 1.75) .. (z.south west);
            \draw[tensor_diagram, postaction={decorate}] (x.north east) -- (y.south west);
            \draw[tensor_diagram, postaction={decorate}] (y.north west) -- (z.south east);
            \draw[tensor_diagram, postaction={decorate}] (z.north west) -- (-0.75, 3.25);
            \draw[tensor_diagram, postaction={decorate}] (z.north east) -- (0.75, 3.25);
            \draw[tensor_diagram, postaction={decorate}] (2.0, -0.75) ..controls (1.75, 0.0) .. (y.south east);
            \draw[tensor_diagram, postaction={decorate}] (y.north east) ..controls (1.75, 2.5) .. (2.0, 3.25);
        \end{scope}
    \end{tikzpicture}
    $\ =\ $
    \begin{tikzpicture}[baseline={([yshift=-1.0ex]current bounding box.center)}]
        \node at (1.25, 0.0) [square, draw] (a) {$Y_a$};
        \node at (0.0, 1.25) [square, draw] (b) {$Y_b$};
        \node at (1.25, 2.5) [square, draw] (c) {$Y_c$};
        \begin{scope}[decoration={
            markings,
            mark=at position 0.75 with {\arrow{>}}}
            ]
            \draw[tensor_diagram, postaction={decorate}] (-0.75, -0.75) .. controls (-0.5, 0.0) .. (b.south west);
            \draw[tensor_diagram, postaction={decorate}] (b.north west) .. controls (-0.5, 2.75) .. (-0.75, 3.25);
            \draw[tensor_diagram, postaction={decorate}] (b.north east) -- (c.south west);
            \draw[tensor_diagram, postaction={decorate}] (a.north west) -- (b.south east);
            \draw[tensor_diagram, postaction={decorate}] (0.5, -0.75) -- (a.south west);
            \draw[tensor_diagram, postaction={decorate}] (2.0, -0.75) -- (a.south east);
            \draw[tensor_diagram, postaction={decorate}] (c.north west) -- (0.5, 3.25);
            \draw[tensor_diagram, postaction={decorate}] (c.north east) -- (2.0, 3.25);
            \draw[tensor_diagram, postaction={decorate}] (a.north east) .. controls (2.1, 0.75) and (2.1, 1.75) .. (c.south east);
        \end{scope}
    \end{tikzpicture}
    \caption{\label{Figure:Axioms2}Tensor diagram for the axiom 2}
\end{figure}
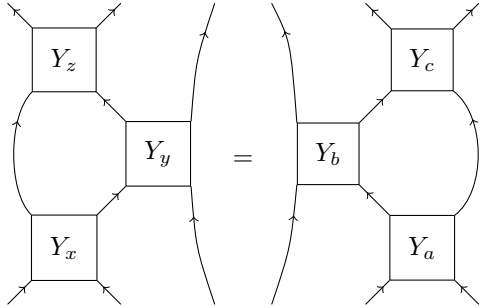

\begin{figure}[ht]
    \begin{tikzpicture}[baseline={([yshift=-0.5ex]current bounding box.center)}]
        \node at (0.0, 0.0) [diamond, draw] (g) {$\Gamma_x$};
        \node at (1.5, -1.5) [square, draw] (x) {$Y_x$};
        \node at (1.5, 1.5) [square, draw] (y) {$\overline{Y_x}$};
        \node at (3.0, -0.75) [square, draw, inner sep=-2pt] (r) {$\overline{R_{x, y_2}}$};
        \node at (3.0, 0.75) [square, draw, inner sep=-2pt] (l) {$\overline{L_{y_1, x}}$};
        \begin{scope}[decoration={
            markings,
            mark=at position 0.75 with {\arrow{>}}}
            ]
            \draw[tensor_diagram, postaction={decorate}] (x.north west) .. controls (0.75, -0.25) and (0.75, 0.25) .. (y.south west);
            \draw[tensor_diagram, postaction={decorate}] (x.north east) .. controls (2.25, -0.75) .. (r.west);
            \draw[tensor_diagram, postaction={decorate}] (l.west) .. controls (2.25, 0.75) .. (y.south east);
            \draw[tensor_diagram, postaction={decorate}] (y.north west) .. controls (0.75, 2.5) and (-0.25, 2.0) .. (g.north);
            \draw[tensor_diagram, postaction={decorate}] (g.south) .. controls (-0.25, -2.0) and (0.75, -2.5) .. (x.south west);
            \draw[tensor_diagram, postaction={decorate}] (2.25, -2.25) .. controls (2.15, -2.05) .. (x.south east);
            \draw[tensor_diagram, postaction={decorate}] (y.north east) .. controls (2.15, 2.05) .. (2.25, 2.25);
            \draw[tensor_diagram, postaction={decorate}] (3.75, 2.25) .. controls (4.0, 2.0) and (4.0, 1.0) .. (l.east);
            \draw[tensor_diagram, postaction={decorate}] (r.east) .. controls (4.0, -1.0) and (4.0, -2.0) .. (3.75, -2.25);
        \end{scope}
    \end{tikzpicture}
    $=$
    \begin{tikzpicture}[baseline={([yshift=-1.0ex]current bounding box.center)}]
        \node at (0.0, 0.0) [square, draw] (e) {$id$};
        \begin{scope}[decoration={
            markings,
            mark=at position 0.75 with {\arrow{>}}}
            ]
            \draw[tensor_diagram, postaction={decorate}] (-0.5, -0.75) -- (e.south west);
            \draw[tensor_diagram, postaction={decorate}] (e.north west) -- (-0.5, 0.75);
            \draw[tensor_diagram, postaction={decorate}] (e.south east) -- (0.5, -0.75);
            \draw[tensor_diagram, postaction={decorate}] (0.5, 0.75) -- (e.north east);
        \end{scope}
    \end{tikzpicture}
    \hfill
    \begin{tikzpicture}[baseline={([yshift=-0.5ex]current bounding box.center)}]
        \node at (1.5, 0.0) [diamond, draw] (g) {$\overline{\Gamma_x}$};
        \node at (0, -1.5) [square, draw] (x) {$\overline{Y_x}$};
        \node at (0, 1.5) [square, draw] (y) {$Y_x$};
        \node at (-1.5, 0.75) [square, draw, inner sep=-2pt] (r) {$R_{y_1, x}$};
        \node at (-1.5, -0.75) [square, draw, inner sep=-2pt] (l) {$L_{x, y_2}$};
        \begin{scope}[decoration={
            markings,
            mark=at position 0.75 with {\arrow{>}}}
            ]
            \draw[tensor_diagram, postaction={decorate}] (x.north east) .. controls (0.75, -0.25) and (0.75, 0.25) .. (y.south east);
            \draw[tensor_diagram, postaction={decorate}] (x.north west) .. controls (-0.75, -0.75) .. (l.east);
            \draw[tensor_diagram, postaction={decorate}] (r.east) .. controls (-0.75, 0.75) .. (y.south west);
            \draw[tensor_diagram, postaction={decorate}] (y.north east) .. controls (0.75, 2.5) and (1.75, 2.0) .. (g.north);
            \draw[tensor_diagram, postaction={decorate}] (g.south) .. controls (1.75, -2.0) and (0.75, -2.5) .. (x.south east);
            \draw[tensor_diagram, postaction={decorate}] (-0.75, -2.25) .. controls (-0.65, -2.05) .. (x.south west);
            \draw[tensor_diagram, postaction={decorate}] (y.north west) .. controls (-0.65, 2.05) .. (-0.75, 2.25);
            \draw[tensor_diagram, postaction={decorate}] (-2.25, 2.25) .. controls (-2.5, 2.0) and (-2.5, 1.0) .. (r.west);
            \draw[tensor_diagram, postaction={decorate}] (l.west) .. controls (-2.5, -1.0) and (-2.5, -2.0) .. (-2.25, -2.25);
        \end{scope}
    \end{tikzpicture}
    $=$
    \begin{tikzpicture}[baseline={([yshift=-1.0ex]current bounding box.center)}]
        \node at (0.0, 0.0) [square, draw] (e) {$id$};
        \begin{scope}[decoration={
            markings,
            mark=at position 0.75 with {\arrow{>}}}
            ]
            \draw[tensor_diagram, postaction={decorate}] (e.south west) -- (-0.5, -0.75);
            \draw[tensor_diagram, postaction={decorate}] (-0.5, 0.75) -- (e.north west);
            \draw[tensor_diagram, postaction={decorate}] (0.5, -0.75) -- (e.south east);
            \draw[tensor_diagram, postaction={decorate}] (e.north east) -- (0.5, 0.75);
        \end{scope}
    \end{tikzpicture}
    \caption{\label{Figure:Axioms3}Tensor diagrams for the axioms 3a and 3b}
\end{figure}
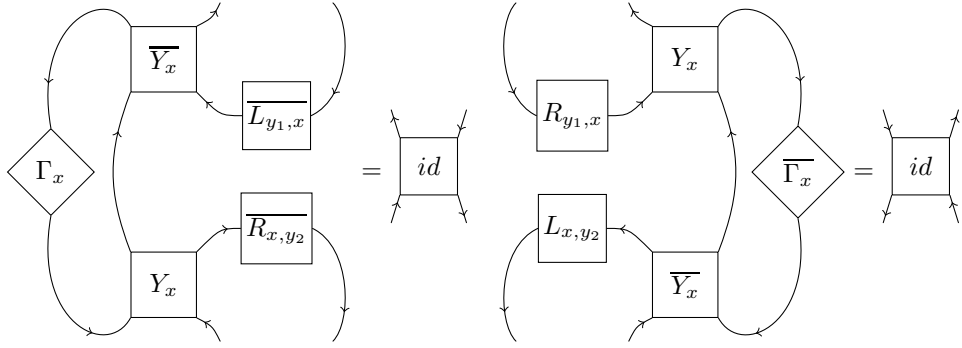

\begin{figure}[ht]
    \begin{tikzpicture}[baseline={([yshift=-0.5ex]current bounding box.center)}]
        \node at (0.0, 0.0) [square, draw] (x) {$Y_x$};
        \node at (-2.5, 0.0) [square, draw, inner sep=-2.5pt] (l1) {$L_{x, y_1}$};
        \node at (-1.25, 0.0) [square, draw, inner sep=-2.5pt] (l2) {$L_{x, y_2}$};        
        \node at (1.25, 0.0) [square, draw, inner sep=-2.5pt] (l3) {$\overline{L_{y_3, x}}$};
        \node at (2.5, 0.0) [square, draw, inner sep=-2.5pt] (l4) {$\overline{L_{y_4, x}}$};
        \begin{scope}[decoration={
            markings,
            mark=at position 0.75 with {\arrow{>}}}
            ]
            \draw[tensor_diagram, postaction={decorate}] (x.north west) .. controls (-0.5, 1.0) and (-1.25, 1.0) .. (l2.north);
            \draw[tensor_diagram, postaction={decorate}] (x.north east) .. controls (0.5, 1.5) and (-2.5, 1.5) .. (l1.north);
            \draw[tensor_diagram, postaction={decorate}] (l4.south) .. controls (2.5, -1.5) and (-0.5, -1.5) .. (x.south west);
            \draw[tensor_diagram, postaction={decorate}] (l3.south) .. controls (1.25, -1.0) and (0.5, -1.0) .. (x.south east);
            \draw[tensor_diagram, postaction={decorate}] (l1.south) -- (-2.5, -1.0);
            \draw[tensor_diagram, postaction={decorate}] (l2.south) -- (-1.25, -1.0);
            \draw[tensor_diagram, postaction={decorate}] (1.25, 1.0) -- (l3.north);
            \draw[tensor_diagram, postaction={decorate}] (2.5, 1.0) -- (l4.north);
        \end{scope}
    \end{tikzpicture}
    $=$
    \begin{tikzpicture}[baseline={([yshift=-0.5ex]current bounding box.center)}]
        \node at (0.0, 0.0) [square, draw] (x) {$Y_x$};
        \node at (-2.5, 0.0) [square, draw, inner sep=-2.5pt] (r1) {$R_{y_3, x}$};
        \node at (-1.25, 0.0) [square, draw, inner sep=-2.5pt] (r2) {$R_{y_4, x}$};
        \node at (1.25, 0.0) [square, draw, inner sep=-2.5pt] (r3) {$\overline{R_{x, y_1}}$};
        \node at (2.5, 0.0) [square, draw, inner sep=-2.5pt] (r4) {$\overline{R_{x, y_2}}$};
        \begin{scope}[decoration={
            markings,
            mark=at position 0.75 with {\arrow{>}}}
            ]
            \draw[tensor_diagram, postaction={decorate}] (x.north east) .. controls (0.5, 1.0) and (1.25, 1.0) .. (r3.north);
            \draw[tensor_diagram, postaction={decorate}] (x.north west) .. controls (-0.5, 1.5) and (2.5, 1.5) .. (r4.north);
            \draw[tensor_diagram, postaction={decorate}] (r2.south) .. controls (-1.25, -1.0) and (-0.5, -1.0) .. (x.south west);
            \draw[tensor_diagram, postaction={decorate}] (r1.south) .. controls (-2.5, -1.5) and (0.5, -1.5) .. (x.south east);
            \draw[tensor_diagram, postaction={decorate}] (r3.south) -- (1.25, -1.0);
            \draw[tensor_diagram, postaction={decorate}] (r4.south) -- (2.5, -1.0);
            \draw[tensor_diagram, postaction={decorate}] (-2.5, 1.0) -- (r1.north);
            \draw[tensor_diagram, postaction={decorate}] (-1.25, 1.0) -- (r2.north);
        \end{scope}
    \end{tikzpicture}

    \begin{tikzpicture}[baseline={([yshift=-0.5ex]current bounding box.center)}]
        \node at (0.0, 0.0) [square, draw] (x) {$\overline{Y_x}$};
        \node at (-2.5, 0.0) [square, draw, inner sep=-2.5pt] (l1) {$L_{x, y_1}$};
        \node at (-1.25, 0.0) [square, draw, inner sep=-2.5pt] (l2) {$L_{x, y_2}$};        
        \node at (1.25, 0.0) [square, draw, inner sep=-2.5pt] (l3) {$\overline{L_{y_3, x}}$};
        \node at (2.5, 0.0) [square, draw, inner sep=-2.5pt] (l4) {$\overline{L_{y_4, x}}$};
        \begin{scope}[decoration={
            markings,
            mark=at position 0.75 with {\arrow{>}}}
            ]
            \draw[tensor_diagram, postaction={decorate}] (x.north west) .. controls (-0.5, 1.0) and (-1.25, 1.0) .. (l2.north);
            \draw[tensor_diagram, postaction={decorate}] (x.north east) .. controls (0.5, 1.5) and (-2.5, 1.5) .. (l1.north);
            \draw[tensor_diagram, postaction={decorate}] (l4.south) .. controls (2.5, -1.5) and (-0.5, -1.5) .. (x.south west);
            \draw[tensor_diagram, postaction={decorate}] (l3.south) .. controls (1.25, -1.0) and (0.5, -1.0) .. (x.south east);
            \draw[tensor_diagram, postaction={decorate}] (l1.south) -- (-2.5, -1.0);
            \draw[tensor_diagram, postaction={decorate}] (l2.south) -- (-1.25, -1.0);
            \draw[tensor_diagram, postaction={decorate}] (1.25, 1.0) -- (l3.north);
            \draw[tensor_diagram, postaction={decorate}] (2.5, 1.0) -- (l4.north);
        \end{scope}
    \end{tikzpicture}
    $=$
    \begin{tikzpicture}[baseline={([yshift=-0.5ex]current bounding box.center)}]
        \node at (0.0, 0.0) [square, draw] (x) {$\overline{Y_x}$};
        \node at (-2.5, 0.0) [square, draw, inner sep=-2.5pt] (r1) {$R_{y_3, x}$};
        \node at (-1.25, 0.0) [square, draw, inner sep=-2.5pt] (r2) {$R_{y_4, x}$};
        \node at (1.25, 0.0) [square, draw, inner sep=-2.5pt] (r3) {$\overline{R_{x, y_1}}$};
        \node at (2.5, 0.0) [square, draw, inner sep=-2.5pt] (r4) {$\overline{R_{x, y_2}}$};
        \begin{scope}[decoration={
            markings,
            mark=at position 0.75 with {\arrow{>}}}
            ]
            \draw[tensor_diagram, postaction={decorate}] (x.north east) .. controls (0.5, 1.0) and (1.25, 1.0) .. (r3.north);
            \draw[tensor_diagram, postaction={decorate}] (x.north west) .. controls (-0.5, 1.5) and (2.5, 1.5) .. (r4.north);
            \draw[tensor_diagram, postaction={decorate}] (r2.south) .. controls (-1.25, -1.0) and (-0.5, -1.0) .. (x.south west);
            \draw[tensor_diagram, postaction={decorate}] (r1.south) .. controls (-2.5, -1.5) and (0.5, -1.5) .. (x.south east);
            \draw[tensor_diagram, postaction={decorate}] (r3.south) -- (1.25, -1.0);
            \draw[tensor_diagram, postaction={decorate}] (r4.south) -- (2.5, -1.0);
            \draw[tensor_diagram, postaction={decorate}] (-2.5, 1.0) -- (r1.north);
            \draw[tensor_diagram, postaction={decorate}] (-1.25, 1.0) -- (r2.north);
        \end{scope}
    \end{tikzpicture}
    \caption{\label{Figure:Axioms4}Tensor diagrams for the axioms 4a -- 4b}
\end{figure}
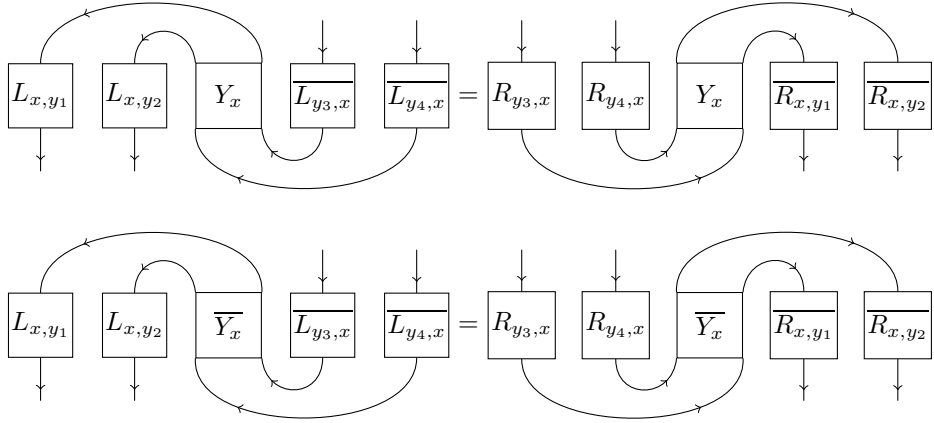

\subsection{Tensor networks for link diagrams}

Let $G$ be a finite group, $\alpha, \beta\in G$, and let $\mathcal{S} = (\mathcal{Y}, \mathcal{L}, \mathcal{R}, A, B)$ be the consistent tensor $G$-system. Next, let $D$ be a diagram of the link $K$ on the 2-sphere $S^2$, and let $\xi\colon C(D)\to G$ be a proper $(\alpha, \beta)$-colouring of the diagram $G$. 

Construct a \emph{tensor network} $T_{D, \xi}$ as follows. Choose a vertical direction on the 2-sphere $S^2$, and use planar isotopes to align the neighbourhoods of all crossings along that direction. This means that the projection of any diagram arc in the neighbourhood of a crossing to that direction should coincide with it. Next, replace each crossing $c\in C(D)$, coloured by the colour $g\in G$, with either $Y_{g}$ (if $c$ is positive) or $\overline{Y_{g}}$ (if $c$ is negative), and each local extrema of the diagram with either $L_{g_1, g_2}$ or $\overline{L_{g_1, g_2}}$ or $R_{g_1, g_2}$ or $\overline{R_{g_1, g_2}}$, where $g_1$ is the colour of the crossing at the beginning of the arc and $g_2$ is the colour of the crossing at the end of the arc (see figure \ref{Figure:ExtremaToTensors}). Finally, connect all these tensors by remaining arcs of the diagram.

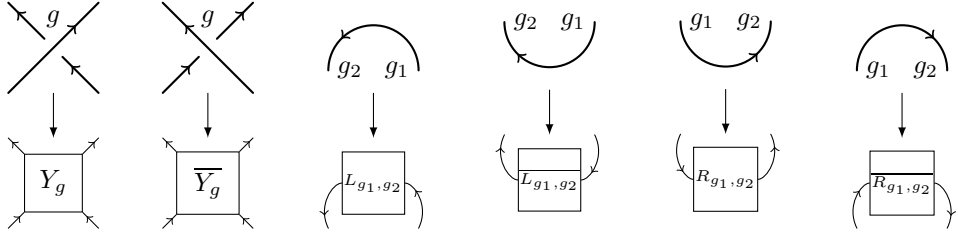
\begin{figure}[ht]
    \begin{tikzpicture}[scale=0.6]
        \begin{scope}[decoration={
            markings,
            mark=at position 0.75 with {\arrow{>}}}
            ]
            \draw[knot_diagram, postaction={decorate}] (-1, -1) -- (1, 1);
            \draw[knot_diagram, postaction={decorate}] (1, -1) -- (0.2, -0.2);
            \draw[knot_diagram, postaction={decorate}] (-0.2, 0.2) -- (-1, 1);
        \end{scope}
        \draw (0.0, 0.25) node[above] {$g$};
        \draw[-latex] (0, -1) -- (0, -2);

        \node at (0, -3) [square, draw, inner sep=2.5pt] (y) {$Y_g$};
        \begin{scope}[decoration={
            markings,
            mark=at position 0.75 with {\arrow{>}}}
            ]
            \draw[tensor_diagram, postaction={decorate}] (-1.0, -4.0) -- (y);
            \draw[tensor_diagram, postaction={decorate}] (1.0, -4.0) -- (y);
            \draw[tensor_diagram, postaction={decorate}] (y) -- (-1.0, -2.0);
            \draw[tensor_diagram, postaction={decorate}] (y) -- (1.0, -2.0);
        \end{scope}
    \end{tikzpicture}
    \hfill
    \begin{tikzpicture}[scale=0.6]
        \begin{scope}[decoration={
            markings,
            mark=at position 0.75 with {\arrow{>}}}
            ]
            \draw[knot_diagram, postaction={decorate}] (1, -1) -- (-1, 1);
            \draw[knot_diagram, postaction={decorate}] (-1, -1) -- (-0.2, -0.2);
            \draw[knot_diagram, postaction={decorate}] (0.2, 0.2) -- (1, 1);
        \end{scope}
        \draw (0.0, 0.25) node[above] {$g$};
        \draw[-latex] (0, -1) -- (0, -2);

        \node at (0, -3) [square, draw, inner sep=2.5pt] (y) {$\overline{Y_g}$};
        \begin{scope}[decoration={
            markings,
            mark=at position 0.75 with {\arrow{>}}}
            ]
            \draw[tensor_diagram, postaction={decorate}] (-1.0, -4.0) -- (y);
            \draw[tensor_diagram, postaction={decorate}] (1.0, -4.0) -- (y);
            \draw[tensor_diagram, postaction={decorate}] (y) -- (-1.0, -2.0);
            \draw[tensor_diagram, postaction={decorate}] (y) -- (1.0, -2.0);
        \end{scope}
    \end{tikzpicture}
    \hfill
    \begin{tikzpicture}[scale=0.6]
        \begin{scope}[decoration={
            markings,
            mark=at position 0.75 with {\arrow{>}}}
            ]
            \draw[knot_diagram, postaction={decorate}] (1, -0.5) arc (0:180:1);
        \end{scope}
        \draw (-1, -0.5) node[right] {$g_2$};
        \draw (1, -0.5) node[left] {$g_1$};
        \draw[-latex] (0, -1) -- (0, -2);

        \node at (0, -3) [square, draw, inner sep=-3pt] (y) {{\tiny$L_{g_1, g_2}$}};
        \begin{scope}[decoration={
            markings,
            mark=at position 0.75 with {\arrow{>}}}
            ]
            \draw[tensor_diagram, postaction={decorate}] (y.west) .. controls (-0.9, -3.0) and (-1.25, -3.5) .. (-1.0, -4.0);
            \draw[tensor_diagram, postaction={decorate}] (1.0, -4.0) .. controls (1.25, -3.5) and (0.9, -3.0) .. (y.east);
        \end{scope}
    \end{tikzpicture}
    \hfill
    \begin{tikzpicture}[scale=0.6, baseline={([yshift=-1.5ex]current bounding box.south)}]
        \begin{scope}[decoration={
            markings,
            mark=at position 0.75 with {\arrow{>}}}
            ]
            \draw[knot_diagram, postaction={decorate}] (1, 0.5) arc (0:-180:1);
        \end{scope}
        \draw (-1, 0.5) node[right] {$g_2$};
        \draw (1, 0.5) node[left] {$g_1$};
        \draw[-latex] (0, -1.0) -- (0, -2.0);

        \node at (0, -3) [square, draw, inner sep=-3pt] (y) {{\tiny$\overline{L_{g_1, g_2}}$}};
        \begin{scope}[decoration={
            markings,
            mark=at position 0.75 with {\arrow{>}}}
            ]
            \draw[tensor_diagram, postaction={decorate}] (y.west) .. controls (-0.9, -3.0) and (-1.25, -2.5) .. (-1.0, -2.0);
            \draw[tensor_diagram, postaction={decorate}] (1.0, -2.0) .. controls (1.25, -2.5) and (0.9, -3.0) .. (y.east);
        \end{scope}
    \end{tikzpicture}
    \hfill
    \begin{tikzpicture}[scale=0.6, baseline={([yshift=-1.5ex]current bounding box.south)}]
        \begin{scope}[decoration={
            markings,
            mark=at position 0.75 with {\arrow{>}}}
            ]
            \draw[knot_diagram, postaction={decorate}] (-1, 0.5) arc (180:360:1);
        \end{scope}
        \draw (-1, 0.5) node[right] {$g_1$};
        \draw (1, 0.5) node[left] {$g_2$};
        \draw[-latex] (0, -1.0) -- (0, -2.0);

        \node at (0, -3) [square, draw, inner sep=-3pt] (y) {{\tiny$R_{g_1, g_2}$}};
        \begin{scope}[decoration={
            markings,
            mark=at position 0.75 with {\arrow{>}}}
            ]
            \draw[tensor_diagram, postaction={decorate}] (-1.0, -2.0) .. controls (-1.25, -2.5) and (-0.9, -3.0) .. (y.west);
            \draw[tensor_diagram, postaction={decorate}] (y.east) .. controls (0.9, -3.0) and (1.25, -2.5) .. (1.0, -2.0);
        \end{scope}
    \end{tikzpicture}
    \hfill
    \begin{tikzpicture}[scale=0.6]
        \begin{scope}[decoration={
            markings,
            mark=at position 0.75 with {\arrow{>}}}
            ]
            \draw[knot_diagram, postaction={decorate}] (-1, -0.5) arc (180:0:1);
        \end{scope}
        \draw (-1, -0.5) node[right] {$g_1$};
        \draw (1, -0.5) node[left] {$g_2$};
        \draw[-latex] (0, -1) -- (0, -2);

        \node at (0, -3) [square, draw, inner sep=-3pt] (y) {{\tiny$\overline{R_{g_1, g_2}}$}};
        \begin{scope}[decoration={
            markings,
            mark=at position 0.75 with {\arrow{>}}}
            ]
            \draw[tensor_diagram, postaction={decorate}] (-1.0, -4.0) .. controls (-1.25, -3.5) and (-0.9, -3.0) .. (y.west);
            \draw[tensor_diagram, postaction={decorate}] (y.east) .. controls (0.9, -3.0) and (1.25, -3.5) .. (1.0, -4.0);
        \end{scope}
    \end{tikzpicture}
    \caption{\label{Figure:ExtremaToTensors}Crossings and local extrema of the diagram and corresponding tensors}
\end{figure}

\begin{example}
    Let $D$ be the diagram of the <<positive trefoil>> knot shown in the figure \ref{Figure:TrefoilTensor} on the left. Consider a proper $(e, e)$-colouring $\xi\colon V(D)\to G$, mapping each crossing of $D$ to the same element $x\in G$ such that $x^3 = e$. Corresponding tensor network $T_{D, \xi}$ shown in the same figure \ref{Figure:TrefoilTensor} on the right.

    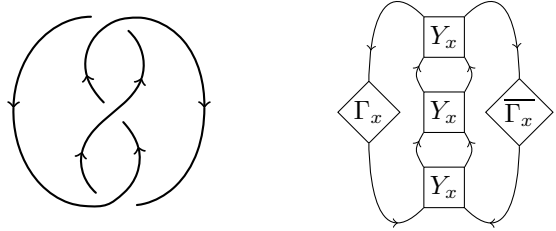
\begin{figure}[ht]
        \begin{tikzpicture}[baseline=(current bounding box.center)]
            \draw[knot_diagram] (2.444, 28.542).. controls (2.498, 28.707) and (2.468, 28.895) .. (2.374, 29.041).. controls (2.344, 29.087) and (2.308, 29.129) .. (2.268, 29.166);
            \draw[knot_diagram, ->] (1.653, 27.561).. controls (1.678, 27.665) and (1.735, 27.759) .. (1.805, 27.84).. controls (1.843, 27.884) and (1.907, 27.945) .. (1.928, 27.964).. controls (1.95, 27.984) and (2.198, 28.194) .. (2.209, 28.205).. controls (2.221, 28.215) and (2.255, 28.249) .. (2.277, 28.272).. controls (2.349, 28.35) and (2.411, 28.44) .. (2.444, 28.542);
            \draw[knot_diagram, ->] (1.834, 27.012).. controls (1.764, 27.075) and (1.709, 27.154) .. (1.675, 27.241).. controls (1.636, 27.342) and (1.627, 27.455) .. (1.653, 27.561);
            \draw[knot_diagram] (3.277, 28.138).. controls (3.276, 27.942) and (3.249, 27.745) .. (3.183, 27.561).. controls (3.117, 27.376) and (3.011, 27.205) .. (2.865, 27.074).. controls (2.756, 26.978) and (2.626, 26.904) .. (2.486, 26.868).. controls (2.451, 26.859) and (2.416, 26.853) .. (2.38, 26.849);
            \draw[knot_diagram, ->] (1.711, 28.583).. controls (1.671, 28.735) and (1.701, 28.902) .. (1.784, 29.035).. controls (1.815, 29.085) and (1.853, 29.132) .. (1.898, 29.17).. controls (2.055, 29.306) and (2.177, 29.345) .. (2.409, 29.342).. controls (2.478, 29.341) and (2.548, 29.324) .. (2.614, 29.301).. controls (2.763, 29.25) and (2.895, 29.156) .. (2.997, 29.037).. controls (3.205, 28.793) and (3.279, 28.459) .. (3.277, 28.138);
            \draw[knot_diagram, ->] (1.939, 28.214).. controls (1.906, 28.248) and (1.875, 28.283) .. (1.846, 28.321).. controls (1.786, 28.399) and (1.735, 28.487) .. (1.711, 28.583);
            \draw[knot_diagram] (2.395, 27.627).. controls (2.368, 27.722) and (2.316, 27.809) .. (2.254, 27.887).. controls (2.236, 27.909) and (2.218, 27.931) .. (2.199, 27.952);
            \draw[knot_diagram, ->] (0.739, 28.137).. controls (0.737, 28.071) and (0.739, 28.005) .. (0.746, 27.94).. controls (0.776, 27.639) and (0.891, 27.342) .. (1.1, 27.124).. controls (1.206, 27.014) and (1.335, 26.925) .. (1.48, 26.876).. controls (1.543, 26.855) and (1.609, 26.842) .. (1.676, 26.837).. controls (1.966, 26.813) and (2.028, 26.847) .. (2.213, 27.017).. controls (2.287, 27.085) and (2.347, 27.173) .. (2.381, 27.269).. controls (2.422, 27.383) and (2.428, 27.51) .. (2.395, 27.627);
            \draw[knot_diagram, ->] (1.769, 29.354).. controls (1.659, 29.354) and (1.548, 29.331) .. (1.446, 29.289).. controls (1.285, 29.224) and (1.144, 29.113) .. (1.034, 28.979).. controls (0.842, 28.744) and (0.746, 28.441) .. (0.739, 28.137);
        \end{tikzpicture}
        \hspace{1.5cm}
        \begin{tikzpicture}[baseline=(current bounding box.center)]
            \node at (0, 0) [draw, square, inner sep=0pt] (y2) {$Y_x$};
            \node at (0, 1) [draw, square, inner sep=0pt] (y3) {$Y_x$};
            \node at (0, -1) [draw, square, inner sep=0pt] (y1) {$Y_x$};
            \node at (-1, 0) [draw, diamond, inner sep=1pt] (gl) {$\Gamma_x$};
            \node at (1, 0) [draw, diamond, inner sep=1pt] (gr) {$\overline{\Gamma_x}$};
            \begin{scope}[decoration={
                markings,
                mark=at position 0.75 with {\arrow{>}}}
                ]
                \draw[tensor_diagram, postaction={decorate}] (y1.north west) .. controls (-0.4, -0.5) .. (y2.south west);
                \draw[tensor_diagram, postaction={decorate}] (y1.north east) .. controls (0.4, -0.5) .. (y2.south east);

                \draw[tensor_diagram, postaction={decorate}] (y2.north west) .. controls (-0.4, 0.5) .. (y3.south west);
                \draw[tensor_diagram, postaction={decorate}] (y2.north east) .. controls (0.4, 0.5) .. (y3.south east);

                \draw[tensor_diagram, postaction={decorate}] (y3.north west) .. controls (-1.0, 2) and (-1.0, 0.5) .. (gl.north);
                \draw[tensor_diagram, postaction={decorate}] (y3.north east) .. controls (1.0, 2) and (1.0, 0.5) .. (gr.north);

                \draw[tensor_diagram, postaction={decorate}] (gl.south) .. controls (-1.0, -0.5) and (-1.0, -2) .. (y1.south west);
                \draw[tensor_diagram, postaction={decorate}] (gr.south) .. controls (1.0, -0.5) and (1.0, -2) .. (y1.south east);
            \end{scope}
        \end{tikzpicture}
        \caption{\label{Figure:TrefoilTensor}Diagram of the <<positive trefoil>> knot (on the left) and corresponding tensor network (on the right)}
    \end{figure}
\end{example}

Note that the tensor diagram $T_{D, \xi}$ has no input and output arrows. So it can be identified with an element of the ring $K$.

Define the value $$t_{\mathcal{S}}(D, \xi) = A^{r(D)}\cdot B^{w(D)}\cdot T_{D, \xi},$$ where $r(D)$ is a total rotation of the tangent vector of $D$ (counted with respect to the orientation of $S^2$), and $w(D)$ is the difference of the number of positive and negative crossings of $D$.

\begin{theorem}
    \label{Theorem:TensorInvariant}
    Let $G$ be a finite group, $\alpha, \beta\in G$, $\mathcal{S} = (\mathcal{Y}, \mathcal{L}, \mathcal{R}, A, B)$ be a consistent tensor $G$-system. Let $D$ be a diagram of an oriented link, $\xi\colon C(D)\to G$ a proper $(\alpha, \beta)$-colouring of $D$. If $D'$ is another diagram of the same link and $\xi'\colon C(D')\to G$ is a correspondence proper $(\alpha, \beta)$-colouring, then $t_{\mathcal{S}}(D, \xi) = t_{\mathcal{S}}(D', \xi')$.
\end{theorem}
\begin{proof}
    To compute $t_{\mathcal{S}}(D, \xi)$ we required that the crossings of the diagram $D$ are aligned along a fixed vertical direction. It's known (see \cite{TRM}) that any two diagrams of the same oriented link with such aligned crossings can be connected by a finite sequence of the following local moves $W_1^r, W_1^l, W_2^r, W_2^l, R_2^l, R_2^r$, $R_1^{+r}, R_1^{-r}, R_1^{+l}, R_1^{-l}$, $R_3$, $\Omega_1, \Omega_2$, $T_1$ and $T_2$ (figures \ref{Figure:LocalOrientedMovesOne}, \ref{Figure:LocalOrientedMovesTwo}, \ref{Figure:LocalOrientedMovesThree}). This is not a minimal set of moves required to transform any diagram into any other. But we will consider this extended set of moves for simplicity.

    \begin{figure}[ht]
        \begin{tikzpicture}[scale=0.85, baseline={([yshift=-1.0ex]current bounding box.center)}]
            \draw[knot_diagram] (1.304, 28.594).. controls (1.276, 28.873) and (1.284, 29.156) .. (1.327, 29.434);
            \draw[knot_diagram, ->] (0.33, 28.519).. controls (0.329, 28.619) and (0.342, 28.72) .. (0.384, 28.81).. controls (0.426, 28.9) and (0.5, 28.978) .. (0.594, 29.01).. controls (0.64, 29.027) and (0.692, 29.032) .. (0.74, 29.022).. controls (0.789, 29.013) and (0.835, 28.989) .. (0.869, 28.953).. controls (0.901, 28.92) and (0.922, 28.878) .. (0.932, 28.833).. controls (0.943, 28.789) and (0.944, 28.743) .. (0.938, 28.697).. controls (0.925, 28.607) and (0.888, 28.522) .. (0.853, 28.438).. controls (0.818, 28.354) and (0.786, 28.266) .. (0.785, 28.175).. controls (0.785, 28.129) and (0.793, 28.083) .. (0.811, 28.042).. controls (0.83, 28.0) and (0.86, 27.963) .. (0.898, 27.939).. controls (0.933, 27.917) and (0.973, 27.907) .. (1.014, 27.906).. controls (1.054, 27.906) and (1.094, 27.917) .. (1.13, 27.935).. controls (1.202, 27.972) and (1.256, 28.039) .. (1.286, 28.114).. controls (1.317, 28.189) and (1.325, 28.271) .. (1.324, 28.352).. controls (1.323, 28.433) and (1.312, 28.513) .. (1.304, 28.594);
            \draw[knot_diagram, ->] (0.365, 27.661).. controls (0.376, 27.848) and (0.372, 28.036) .. (0.354, 28.223).. controls (0.344, 28.322) and (0.33, 28.42) .. (0.33, 28.519);
        \end{tikzpicture}
        $\xleftarrow{W_1^r}$
        \begin{tikzpicture}[scale=0.7, baseline={([yshift=-1.0ex]current bounding box.center)}]
            \begin{scope}[decoration={
                markings,
                mark=at position 0.75 with {\arrow{>}}}
                ]
                \draw[knot_diagram, postaction={decorate}] (0, -1) -- (0, 1);
            \end{scope}
        \end{tikzpicture}
        $\xrightarrow{W_1^l}$
        \begin{tikzpicture}[scale=0.85, baseline={([yshift=-1.0ex]current bounding box.center)}]
            \draw[knot_diagram] (2.265, 28.599).. controls (2.297, 28.878) and (2.307, 29.159) .. (2.294, 29.44);
            \draw[knot_diagram, ->] (3.31, 28.459).. controls (3.326, 28.541) and (3.335, 28.624) .. (3.323, 28.706).. controls (3.312, 28.788) and (3.278, 28.869) .. (3.218, 28.926).. controls (3.188, 28.955) and (3.152, 28.977) .. (3.112, 28.989).. controls (3.073, 29.001) and (3.03, 29.003) .. (2.99, 28.993).. controls (2.941, 28.981) and (2.897, 28.951) .. (2.864, 28.913).. controls (2.832, 28.874) and (2.81, 28.827) .. (2.797, 28.779).. controls (2.772, 28.681) and (2.782, 28.578) .. (2.793, 28.478).. controls (2.803, 28.377) and (2.813, 28.275) .. (2.788, 28.177).. controls (2.775, 28.128) and (2.753, 28.081) .. (2.721, 28.043).. controls (2.688, 28.004) and (2.645, 27.974) .. (2.596, 27.962).. controls (2.561, 27.953) and (2.525, 27.954) .. (2.49, 27.962).. controls (2.456, 27.969) and (2.423, 27.985) .. (2.394, 28.006).. controls (2.337, 28.048) and (2.296, 28.11) .. (2.274, 28.177).. controls (2.228, 28.312) and (2.249, 28.458) .. (2.265, 28.599);
            \draw[knot_diagram, ->] (3.256, 27.667).. controls (3.221, 27.848) and (3.221, 28.036) .. (3.256, 28.217).. controls (3.272, 28.298) and (3.294, 28.378) .. (3.31, 28.459);
        \end{tikzpicture}
        \hfill
        \begin{tikzpicture}[scale=0.85, baseline={([yshift=-1.0ex]current bounding box.center)}]
            \draw[knot_diagram, <-] (1.304, 28.594).. controls (1.276, 28.873) and (1.284, 29.156) .. (1.327, 29.434);
            \draw[knot_diagram, <-] (0.33, 28.519).. controls (0.329, 28.619) and (0.342, 28.72) .. (0.384, 28.81).. controls (0.426, 28.9) and (0.5, 28.978) .. (0.594, 29.01).. controls (0.64, 29.027) and (0.692, 29.032) .. (0.74, 29.022).. controls (0.789, 29.013) and (0.835, 28.989) .. (0.869, 28.953).. controls (0.901, 28.92) and (0.922, 28.878) .. (0.932, 28.833).. controls (0.943, 28.789) and (0.944, 28.743) .. (0.938, 28.697).. controls (0.925, 28.607) and (0.888, 28.522) .. (0.853, 28.438).. controls (0.818, 28.354) and (0.786, 28.266) .. (0.785, 28.175).. controls (0.785, 28.129) and (0.793, 28.083) .. (0.811, 28.042).. controls (0.83, 28.0) and (0.86, 27.963) .. (0.898, 27.939).. controls (0.933, 27.917) and (0.973, 27.907) .. (1.014, 27.906).. controls (1.054, 27.906) and (1.094, 27.917) .. (1.13, 27.935).. controls (1.202, 27.972) and (1.256, 28.039) .. (1.286, 28.114).. controls (1.317, 28.189) and (1.325, 28.271) .. (1.324, 28.352).. controls (1.323, 28.433) and (1.312, 28.513) .. (1.304, 28.594);
            \draw[knot_diagram] (0.365, 27.661).. controls (0.376, 27.848) and (0.372, 28.036) .. (0.354, 28.223).. controls (0.344, 28.322) and (0.33, 28.42) .. (0.33, 28.519);
        \end{tikzpicture}
        $\xleftarrow{W_2^l}$
        \begin{tikzpicture}[scale=0.7, baseline={([yshift=-1.0ex]current bounding box.center)}]
            \begin{scope}[decoration={
                markings,
                mark=at position 0.75 with {\arrow{>}}}
                ]
                \draw[knot_diagram, postaction={decorate}] (0, 1) -- (0, -1);
            \end{scope}
        \end{tikzpicture}
        $\xrightarrow{W_2^r}$
        \begin{tikzpicture}[scale=0.85, baseline={([yshift=-1.0ex]current bounding box.center)}]
            \draw[knot_diagram, <-] (2.265, 28.599).. controls (2.297, 28.878) and (2.307, 29.159) .. (2.294, 29.44);
            \draw[knot_diagram, <-] (3.31, 28.459).. controls (3.326, 28.541) and (3.335, 28.624) .. (3.323, 28.706).. controls (3.312, 28.788) and (3.278, 28.869) .. (3.218, 28.926).. controls (3.188, 28.955) and (3.152, 28.977) .. (3.112, 28.989).. controls (3.073, 29.001) and (3.03, 29.003) .. (2.99, 28.993).. controls (2.941, 28.981) and (2.897, 28.951) .. (2.864, 28.913).. controls (2.832, 28.874) and (2.81, 28.827) .. (2.797, 28.779).. controls (2.772, 28.681) and (2.782, 28.578) .. (2.793, 28.478).. controls (2.803, 28.377) and (2.813, 28.275) .. (2.788, 28.177).. controls (2.775, 28.128) and (2.753, 28.081) .. (2.721, 28.043).. controls (2.688, 28.004) and (2.645, 27.974) .. (2.596, 27.962).. controls (2.561, 27.953) and (2.525, 27.954) .. (2.49, 27.962).. controls (2.456, 27.969) and (2.423, 27.985) .. (2.394, 28.006).. controls (2.337, 28.048) and (2.296, 28.11) .. (2.274, 28.177).. controls (2.228, 28.312) and (2.249, 28.458) .. (2.265, 28.599);
            \draw[knot_diagram] (3.256, 27.667).. controls (3.221, 27.848) and (3.221, 28.036) .. (3.256, 28.217).. controls (3.272, 28.298) and (3.294, 28.378) .. (3.31, 28.459);
        \end{tikzpicture}
        \hfill
        \begin{tikzpicture}[scale=0.85, baseline={([yshift=-1.0ex]current bounding box.center)}]
            \draw[knot_diagram] (5.813, 28.455).. controls (5.803, 28.707) and (5.906, 28.962) .. (6.088, 29.136).. controls (6.27, 29.311) and (6.529, 29.403) .. (6.78, 29.382);
            \draw[knot_diagram, ->] (6.794, 27.6).. controls (6.552, 27.576) and (6.302, 27.656) .. (6.119, 27.816).. controls (5.936, 27.976) and (5.822, 28.212) .. (5.813, 28.455);
            \draw[knot_diagram] (6.07, 29.271).. controls (5.916, 29.356) and (5.737, 29.396) .. (5.561, 29.382);
            \draw[knot_diagram] (6.529, 28.455).. controls (6.539, 28.707) and (6.436, 28.962) .. (6.254, 29.136);
            \draw[knot_diagram, ->] (6.246, 27.837).. controls (6.415, 27.996) and (6.52, 28.223) .. (6.529, 28.455);
            \draw[knot_diagram] (5.548, 27.6).. controls (5.739, 27.581) and (5.936, 27.627) .. (6.1, 27.727);
        \end{tikzpicture}
        $\xleftarrow{R_2^l}$
        \begin{tikzpicture}[scale=0.85, baseline={([yshift=-1.0ex]current bounding box.center)}]
            \draw[knot_diagram] (4.257, 28.993).. controls (4.212, 29.131) and (4.148, 29.262) .. (4.066, 29.382);
            \draw[knot_diagram, ->] (4.071, 27.627).. controls (4.257, 27.904) and (4.351, 28.242) .. (4.335, 28.576).. controls (4.328, 28.718) and (4.302, 28.859) .. (4.257, 28.993);
            \draw[knot_diagram] (4.803, 28.99).. controls (4.849, 29.129) and (4.913, 29.261) .. (4.995, 29.382);
            \draw[knot_diagram, ->] (4.991, 27.627).. controls (4.804, 27.904) and (4.71, 28.242) .. (4.726, 28.576).. controls (4.733, 28.717) and (4.759, 28.856) .. (4.803, 28.99);
        \end{tikzpicture}
        $\xrightarrow{R_2^r}$
        \begin{tikzpicture}[scale=0.85, baseline={([yshift=-1.0ex]current bounding box.center)}]
            \draw[knot_diagram] (8.086, 29.264).. controls (8.243, 29.354) and (8.427, 29.397) .. (8.607, 29.382);
            \draw[knot_diagram] (7.64, 28.455).. controls (7.63, 28.707) and (7.733, 28.962) .. (7.915, 29.136);
            \draw[knot_diagram, ->] (7.914, 27.845).. controls (7.75, 28.004) and (7.649, 28.227) .. (7.64, 28.455);
            \draw[knot_diagram] (8.621, 27.6).. controls (8.442, 27.582) and (8.258, 27.622) .. (8.1, 27.708);
            \draw[knot_diagram] (8.356, 28.455).. controls (8.366, 28.707) and (8.263, 28.962) .. (8.081, 29.136).. controls (7.899, 29.311) and (7.64, 29.403) .. (7.389, 29.382);
            \draw[knot_diagram, ->] (7.375, 27.6).. controls (7.617, 27.576) and (7.867, 27.656) .. (8.05, 27.816).. controls (8.233, 27.976) and (8.346, 28.212) .. (8.356, 28.455);
        \end{tikzpicture}
        \caption{\label{Figure:LocalOrientedMovesOne}Moves $W_1^r, W_1^l, W_2^l, W_2^r, R_2^l$ and $R_2^r$}
    \end{figure}
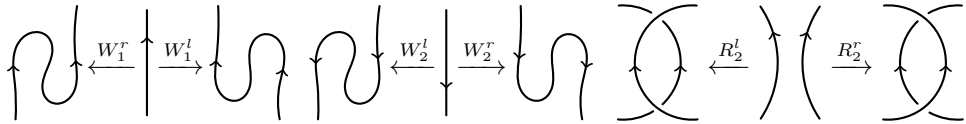

    \begin{figure}[ht]
        \begin{tikzpicture}[scale=0.85, baseline={([yshift=-1.0ex]current bounding box.center)}]
            \draw[knot_diagram] (11.908, 28.436).. controls (11.918, 28.343) and (11.967, 28.255) .. (12.04, 28.197).. controls (12.113, 28.139) and (12.209, 28.112) .. (12.302, 28.123).. controls (12.362, 28.13) and (12.421, 28.154) .. (12.472, 28.187).. controls (12.523, 28.221) and (12.567, 28.265) .. (12.603, 28.314).. controls (12.675, 28.413) and (12.717, 28.531) .. (12.748, 28.649).. controls (12.809, 28.874) and (12.836, 29.109) .. (12.831, 29.342);
            \draw[knot_diagram, ->] (12.635, 28.606).. controls (12.594, 28.667) and (12.543, 28.721) .. (12.481, 28.761).. controls (12.406, 28.809) and (12.315, 28.835) .. (12.226, 28.826).. controls (12.133, 28.817) and (12.045, 28.769) .. (11.986, 28.697).. controls (11.927, 28.625) and (11.898, 28.529) .. (11.908, 28.436);
            \draw[knot_diagram] (12.786, 28.047).. controls (12.778, 28.145) and (12.768, 28.244) .. (12.745, 28.339);
            \draw[knot_diagram, ->] (12.812, 27.506).. controls (12.81, 27.686) and (12.801, 27.867) .. (12.786, 28.047);
        \end{tikzpicture}
        $\xleftarrow{R_1^{+l}}$
        \begin{tikzpicture}[scale=0.75, baseline={([yshift=-1.0ex]current bounding box.center)}]
            \begin{scope}[decoration={
                markings,
                mark=at position 0.75 with {\arrow{>}}}
                ]
                \draw[knot_diagram, postaction={decorate}] (0, -1) -- (0, 1);
            \end{scope}
        \end{tikzpicture}
        $\xrightarrow{R_1^{+r}}$
        \begin{tikzpicture}[scale=0.85, baseline={([yshift=-1.0ex]current bounding box.center)}]
            \draw[knot_diagram] (10.755, 28.689).. controls (10.754, 28.693) and (10.753, 28.698) .. (10.751, 28.703).. controls (10.691, 28.928) and (10.663, 29.162) .. (10.669, 29.396);
            \draw[knot_diagram] (11.591, 28.49).. controls (11.581, 28.397) and (11.533, 28.309) .. (11.46, 28.251).. controls (11.387, 28.193) and (11.29, 28.166) .. (11.198, 28.177).. controls (11.137, 28.184) and (11.079, 28.207) .. (11.028, 28.241).. controls (10.977, 28.275) and (10.933, 28.319) .. (10.897, 28.368).. controls (10.884, 28.386) and (10.871, 28.405) .. (10.86, 28.424);
            \draw[knot_diagram, ->] (10.713, 28.101).. controls (10.728, 28.279) and (10.751, 28.463) .. (10.839, 28.618).. controls (10.883, 28.696) and (10.943, 28.766) .. (11.019, 28.814).. controls (11.094, 28.863) and (11.185, 28.889) .. (11.274, 28.88).. controls (11.367, 28.871) and (11.455, 28.823) .. (11.514, 28.751).. controls (11.573, 28.679) and (11.601, 28.582) .. (11.591, 28.49);
            \draw[knot_diagram, ->] (10.688, 27.559).. controls (10.69, 27.74) and (10.698, 27.921) .. (10.713, 28.101);
        \end{tikzpicture}
        \hfill
        \begin{tikzpicture}[scale=0.85, baseline={([yshift=-1.0ex]current bounding box.center)}]
            \draw[knot_diagram] (13.989, 28.635).. controls (13.99, 28.64) and (13.992, 28.644) .. (13.993, 28.649).. controls (14.053, 28.874) and (14.081, 29.109) .. (14.075, 29.342);
            \draw[knot_diagram] (13.153, 28.436).. controls (13.163, 28.343) and (13.211, 28.255) .. (13.284, 28.197).. controls (13.357, 28.139) and (13.454, 28.112) .. (13.546, 28.123).. controls (13.607, 28.13) and (13.665, 28.154) .. (13.716, 28.187).. controls (13.767, 28.221) and (13.811, 28.265) .. (13.847, 28.314).. controls (13.86, 28.332) and (13.873, 28.351) .. (13.884, 28.37);
            \draw[knot_diagram, ->] (14.031, 28.047).. controls (14.016, 28.225) and (13.993, 28.409) .. (13.905, 28.565).. controls (13.861, 28.643) and (13.801, 28.712) .. (13.725, 28.761).. controls (13.65, 28.809) and (13.559, 28.835) .. (13.47, 28.826).. controls (13.377, 28.817) and (13.289, 28.769) .. (13.23, 28.697).. controls (13.171, 28.625) and (13.143, 28.529) .. (13.153, 28.436);
            \draw[knot_diagram, ->] (14.056, 27.506).. controls (14.054, 27.686) and (14.046, 27.867) .. (14.031, 28.047);
        \end{tikzpicture}
        $\xleftarrow{R_1^{-l}}$
        \begin{tikzpicture}[scale=0.75, baseline={([yshift=-1.0ex]current bounding box.center)}]
            \begin{scope}[decoration={
                markings,
                mark=at position 0.75 with {\arrow{>}}}
                ]
                \draw[knot_diagram, postaction={decorate}] (0, -1) -- (0, 1);
            \end{scope}
        \end{tikzpicture}
        $\xrightarrow{R_1^{-r}}$
        \begin{tikzpicture}[scale=0.85, baseline={([yshift=-1.0ex]current bounding box.center)}]
            \draw[knot_diagram] (10.188, 28.471).. controls (10.178, 28.378) and (10.129, 28.29) .. (10.056, 28.232).. controls (9.984, 28.174) and (9.887, 28.147) .. (9.794, 28.158).. controls (9.734, 28.165) and (9.675, 28.188) .. (9.624, 28.222).. controls (9.573, 28.256) and (9.53, 28.3) .. (9.493, 28.349).. controls (9.421, 28.448) and (9.379, 28.565) .. (9.348, 28.684).. controls (9.288, 28.909) and (9.26, 29.143) .. (9.266, 29.377);
            \draw[knot_diagram, ->] (9.461, 28.641).. controls (9.502, 28.702) and (9.553, 28.756) .. (9.615, 28.795).. controls (9.69, 28.844) and (9.781, 28.87) .. (9.87, 28.861).. controls (9.963, 28.852) and (10.052, 28.804) .. (10.11, 28.732).. controls (10.169, 28.66) and (10.198, 28.563) .. (10.188, 28.471);
            \draw[knot_diagram] (9.31, 28.082).. controls (9.318, 28.18) and (9.328, 28.279) .. (9.351, 28.374);
            \draw[knot_diagram, ->] (9.285, 27.54).. controls (9.287, 27.721) and (9.295, 27.902) .. (9.31, 28.082);
        \end{tikzpicture}
        \hfill
        \begin{tikzpicture}[scale=0.85, baseline={([yshift=-1.0ex]current bounding box.center)}]
            \draw[knot_diagram] (15.719, 28.908).. controls (15.797, 29.099) and (15.84, 29.304) .. (15.846, 29.51);
            \draw[knot_diagram, ->](15.225, 28.189).. controls (15.311, 28.281) and (15.394, 28.374) .. (15.469, 28.475).. controls (15.569, 28.608) and (15.656, 28.753) .. (15.719, 28.908);
            \draw[knot_diagram, ->] (14.589, 27.474).. controls (14.671, 27.586) and (14.759, 27.694) .. (14.851, 27.797).. controls (14.972, 27.931) and (15.102, 28.058) .. (15.225, 28.189);
            \draw[knot_diagram] (15.147, 29.275).. controls (15.188, 29.351) and (15.224, 29.429) .. (15.254, 29.51);
            \draw[knot_diagram, ->] (14.76, 28.181).. controls (14.721, 28.23) and (14.686, 28.284) .. (14.663, 28.343).. controls (14.64, 28.401) and (14.629, 28.466) .. (14.636, 28.528).. controls (14.646, 28.609) and (14.686, 28.683) .. (14.735, 28.747).. controls (14.785, 28.812) and (14.844, 28.868) .. (14.899, 28.927).. controls (14.996, 29.032) and (15.079, 29.15) .. (15.147, 29.275);
            \draw[knot_diagram, ->] (14.925, 28.003).. controls (14.869, 28.062) and (14.811, 28.118) .. (14.76, 28.181);
            \draw[knot_diagram] (15.216, 27.454).. controls (15.192, 27.603) and (15.132, 27.745) .. (15.043, 27.867);
            \draw[knot_diagram] (14.896, 29.099).. controls (14.779, 29.212) and (14.668, 29.364) .. (14.598, 29.51);
            \draw[knot_diagram] (15.209, 28.832).. controls (15.161, 28.878) and (15.11, 28.92) .. (15.059, 28.962);
            \draw[knot_diagram, ->] (15.417, 28.591).. controls (15.352, 28.683) and (15.291, 28.755) .. (15.209, 28.832);
            \draw[knot_diagram] (15.719, 28.023).. controls (15.671, 28.154) and (15.613, 28.28) .. (15.543, 28.399);
            \draw[knot_diagram, ->] (15.846, 27.468).. controls (15.826, 27.657) and (15.783, 27.844) .. (15.719, 28.023);
        \end{tikzpicture}
        $\xrightarrow{R_3}$
        \begin{tikzpicture}[scale=0.85, baseline={([yshift=-1.0ex]current bounding box.center)}]
            \draw[knot_diagram] (16.475, 29.213).. controls (16.455, 29.311) and (16.444, 29.41) .. (16.441, 29.51);
            \draw[knot_diagram, ->] (16.74, 28.585).. controls (16.673, 28.687) and (16.614, 28.795) .. (16.568, 28.908).. controls (16.527, 29.007) and (16.496, 29.109) .. (16.475, 29.213);
            \draw[knot_diagram] (17.062, 28.189).. controls (16.997, 28.258) and (16.934, 28.329) .. (16.875, 28.402);
            \draw[knot_diagram, ->] (17.236, 28.008).. controls (17.178, 28.068) and (17.119, 28.128) .. (17.062, 28.189);
            \draw[knot_diagram] (17.698, 27.474).. controls (17.616, 27.586) and (17.528, 27.694) .. (17.435, 27.797).. controls (17.415, 27.82) and (17.394, 27.843) .. (17.373, 27.865);
            \draw[knot_diagram] (17.244, 29.107).. controls (17.157, 29.231) and (17.086, 29.367) .. (17.033, 29.51);
            \draw[knot_diagram] (17.551, 28.747).. controls (17.502, 28.812) and (17.443, 28.868) .. (17.388, 28.927).. controls (17.382, 28.934) and (17.375, 28.941) .. (17.369, 28.948);
            \draw[knot_diagram, ->] (17.527, 28.181).. controls (17.566, 28.23) and (17.601, 28.284) .. (17.624, 28.343).. controls (17.647, 28.401) and (17.658, 28.466) .. (17.651, 28.528).. controls (17.641, 28.609) and (17.601, 28.683) .. (17.551, 28.747);
            \draw[knot_diagram, ->] (17.07, 27.454).. controls (17.098, 27.617) and (17.167, 27.773) .. (17.271, 27.901).. controls (17.35, 28.0) and (17.448, 28.082) .. (17.527, 28.181);
            \draw[knot_diagram] (17.497, 29.213).. controls (17.573, 29.303) and (17.638, 29.403) .. (17.689, 29.51);
            \draw[knot_diagram, ->] (16.577, 28.046).. controls (16.637, 28.209) and (16.716, 28.364) .. (16.812, 28.509).. controls (16.888, 28.626) and (16.976, 28.736) .. (17.078, 28.832).. controls (17.183, 28.932) and (17.303, 29.016) .. (17.407, 29.117).. controls (17.438, 29.148) and (17.468, 29.18) .. (17.497, 29.213);
            \draw[knot_diagram, ->] (16.441, 27.468).. controls (16.461, 27.665) and (16.507, 27.86) .. (16.577, 28.046);
        \end{tikzpicture}
        \caption{\label{Figure:LocalOrientedMovesTwo}Moves $R_1^{+l}, R_1^{+r}, R_1^{-l}, R_1^{-r}$ and $R_3$}
    \end{figure}
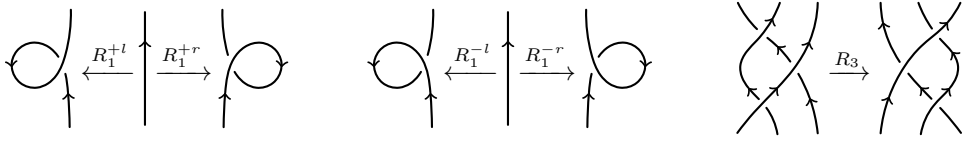

    \begin{figure}[ht]
        \begin{tikzpicture}[scale=0.85, baseline={([yshift=-1.0ex]current bounding box.center)}]
            \draw[knot_diagram] (4.188, 26.112).. controls (4.142, 26.249) and (4.078, 26.381) .. (3.997, 26.5);
            \draw[knot_diagram, <-] (4.2, 25.167).. controls (4.251, 25.337) and (4.274, 25.517) .. (4.265, 25.694).. controls (4.259, 25.836) and (4.232, 25.977) .. (4.188, 26.112);
            \draw[knot_diagram] (4.001, 24.745).. controls (4.088, 24.875) and (4.155, 25.018) .. (4.2, 25.167);
            \draw[knot_diagram] (4.733, 26.108).. controls (4.779, 26.247) and (4.843, 26.38) .. (4.925, 26.5);
            \draw[knot_diagram, ->] (4.719, 25.178).. controls (4.67, 25.345) and (4.648, 25.52) .. (4.657, 25.694).. controls (4.663, 25.835) and (4.69, 25.975) .. (4.733, 26.108);
            \draw[knot_diagram] (4.921, 24.745).. controls (4.832, 24.878) and (4.764, 25.025) .. (4.719, 25.178);
        \end{tikzpicture}
        $\xrightarrow{\Omega_1}$
        \begin{tikzpicture}[scale=0.85, baseline={([yshift=-1.0ex]current bounding box.center)}]
            \draw[knot_diagram] (5.554, 25.178).. controls (5.503, 25.194) and (5.447, 25.191) .. (5.397, 25.171).. controls (5.348, 25.15) and (5.309, 25.11) .. (5.281, 25.064).. controls (5.219, 24.961) and (5.249, 24.809) .. (5.261, 24.757);
            \draw[knot_diagram, ->] (6.365, 25.602).. controls (6.369, 25.424) and (6.412, 25.243) .. (6.372, 25.069).. controls (6.352, 24.982) and (6.31, 24.898) .. (6.241, 24.841).. controls (6.207, 24.812) and (6.167, 24.791) .. (6.123, 24.78).. controls (6.08, 24.769) and (6.034, 24.769) .. (5.991, 24.782).. controls (5.943, 24.797) and (5.901, 24.827) .. (5.866, 24.862).. controls (5.831, 24.898) and (5.802, 24.939) .. (5.773, 24.98).. controls (5.744, 25.021) and (5.714, 25.062) .. (5.679, 25.097).. controls (5.644, 25.133) and (5.602, 25.163) .. (5.554, 25.178);
            \draw[knot_diagram, ->] (5.658, 26.18).. controls (5.692, 26.209) and (5.72, 26.243) .. (5.749, 26.276).. controls (5.777, 26.31) and (5.806, 26.343) .. (5.839, 26.373).. controls (5.871, 26.402) and (5.908, 26.428) .. (5.95, 26.442).. controls (5.994, 26.456) and (6.042, 26.458) .. (6.087, 26.447).. controls (6.132, 26.437) and (6.174, 26.416) .. (6.211, 26.387).. controls (6.283, 26.33) and (6.329, 26.244) .. (6.353, 26.155).. controls (6.4, 25.976) and (6.361, 25.787) .. (6.365, 25.602);
            \draw[knot_diagram, ->] (5.258, 26.504).. controls (5.247, 26.452) and (5.232, 26.317) .. (5.278, 26.231).. controls (5.303, 26.184) and (5.343, 26.146) .. (5.391, 26.125).. controls (5.439, 26.105) and (5.494, 26.101) .. (5.545, 26.115).. controls (5.587, 26.127) and (5.625, 26.152) .. (5.658, 26.18);
            \draw[knot_diagram] (5.734, 26.382).. controls (5.698, 26.423) and (5.659, 26.463) .. (5.618, 26.499);
            \draw[knot_diagram] (5.947, 26.01).. controls (5.921, 26.086) and (5.886, 26.159) .. (5.845, 26.227);
            \draw[knot_diagram, ->] (5.835, 25.033).. controls (5.954, 25.224) and (6.018, 25.45) .. (6.01, 25.675).. controls (6.007, 25.789) and (5.985, 25.902) .. (5.947, 26.01);
            \draw[knot_diagram] (5.589, 24.754).. controls (5.638, 24.794) and (5.684, 24.838) .. (5.726, 24.886);
        \end{tikzpicture}
        \hfill
        \begin{tikzpicture}[scale=0.85, baseline={([yshift=-1.0ex]current bounding box.center)}]
            \draw[knot_diagram] (4.188, 26.112).. controls (4.142, 26.249) and (4.078, 26.381) .. (3.997, 26.5);
            \draw[knot_diagram, ->] (4.2, 25.167).. controls (4.251, 25.337) and (4.274, 25.517) .. (4.265, 25.694).. controls (4.259, 25.836) and (4.232, 25.977) .. (4.188, 26.112);
            \draw[knot_diagram] (4.001, 24.745).. controls (4.088, 24.875) and (4.155, 25.018) .. (4.2, 25.167);
            \draw[knot_diagram] (4.733, 26.108).. controls (4.779, 26.247) and (4.843, 26.38) .. (4.925, 26.5);
            \draw[knot_diagram, <-] (4.719, 25.178).. controls (4.67, 25.345) and (4.648, 25.52) .. (4.657, 25.694).. controls (4.663, 25.835) and (4.69, 25.975) .. (4.733, 26.108);
            \draw[knot_diagram] (4.921, 24.745).. controls (4.832, 24.878) and (4.764, 25.025) .. (4.719, 25.178);
        \end{tikzpicture}
        $\xrightarrow{\Omega_2}$
        \begin{tikzpicture}[scale=0.85, baseline={([yshift=-1.0ex]current bounding box.center)}]
            \draw[knot_diagram] (8.129, 25.243).. controls (8.202, 25.241) and (8.277, 25.21) .. (8.321, 25.152).. controls (8.425, 25.011) and (8.403, 24.898) .. (8.381, 24.761);
            \draw[knot_diagram, ->] (7.315, 25.666).. controls (7.313, 25.565) and (7.293, 25.465) .. (7.278, 25.364).. controls (7.264, 25.264) and (7.254, 25.162) .. (7.271, 25.062).. controls (7.289, 24.963) and (7.337, 24.865) .. (7.42, 24.806).. controls (7.461, 24.776) and (7.51, 24.757) .. (7.56, 24.754).. controls (7.611, 24.75) and (7.663, 24.763) .. (7.704, 24.792).. controls (7.735, 24.812) and (7.759, 24.841) .. (7.779, 24.872).. controls (7.798, 24.903) and (7.813, 24.936) .. (7.828, 24.97).. controls (7.856, 25.038) and (7.882, 25.109) .. (7.932, 25.162).. controls (7.983, 25.215) and (8.056, 25.245) .. (8.129, 25.243);
            \draw[knot_diagram, ->] (8.091, 26.121).. controls (8.031, 26.123) and (7.973, 26.147) .. (7.929, 26.188).. controls (7.886, 26.229) and (7.858, 26.284) .. (7.83, 26.337).. controls (7.802, 26.389) and (7.77, 26.443) .. (7.72, 26.476).. controls (7.681, 26.503) and (7.633, 26.515) .. (7.585, 26.513).. controls (7.538, 26.51) and (7.492, 26.494) .. (7.452, 26.467).. controls (7.373, 26.414) and (7.324, 26.325) .. (7.303, 26.233).. controls (7.282, 26.14) and (7.285, 26.044) .. (7.294, 25.949).. controls (7.303, 25.855) and (7.317, 25.761) .. (7.315, 25.666);
            \draw[knot_diagram, ->] (8.342, 26.505).. controls (8.344, 26.417) and (8.384, 26.274) .. (8.256, 26.178).. controls (8.208, 26.143) and (8.15, 26.12) .. (8.091, 26.121);
            \draw[knot_diagram] (7.906, 26.376).. controls (7.953, 26.422) and (8.004, 26.463) .. (8.059, 26.499);
            \draw[knot_diagram] (7.647, 25.94).. controls (7.666, 26.007) and (7.693, 26.073) .. (7.726, 26.135).. controls (7.742, 26.166) and (7.761, 26.196) .. (7.78, 26.225);
            \draw[knot_diagram, ->] (7.783, 25.095).. controls (7.769, 25.117) and (7.755, 25.139) .. (7.743, 25.162).. controls (7.66, 25.313) and (7.612, 25.484) .. (7.609, 25.656).. controls (7.608, 25.752) and (7.621, 25.848) .. (7.647, 25.94);
            \draw[knot_diagram] (8.116, 24.751).. controls (8.104, 24.759) and (8.093, 24.767) .. (8.081, 24.776).. controls (8.015, 24.826) and (7.953, 24.882) .. (7.898, 24.944);
        \end{tikzpicture}
        \hfill
        \begin{tikzpicture}[scale=0.85, baseline={([yshift=-1.0ex]current bounding box.center)}]
            \draw[knot_diagram] (10.738, 25.821).. controls (10.71, 25.748) and (10.66, 25.685) .. (10.602, 25.632).. controls (10.544, 25.579) and (10.478, 25.536) .. (10.412, 25.493).. controls (10.346, 25.451) and (10.28, 25.408) .. (10.222, 25.356).. controls (10.163, 25.303) and (10.112, 25.241) .. (10.083, 25.169).. controls (10.055, 25.101) and (10.048, 25.026) .. (10.061, 24.955).. controls (10.075, 24.883) and (10.109, 24.816) .. (10.159, 24.763);
            \draw[knot_diagram, ->] (9.98, 26.205).. controls (9.955, 26.236) and (9.931, 26.27) .. (9.921, 26.309).. controls (9.913, 26.337) and (9.913, 26.367) .. (9.92, 26.396).. controls (9.927, 26.424) and (9.941, 26.451) .. (9.959, 26.474).. controls (9.996, 26.52) and (10.051, 26.549) .. (10.109, 26.56).. controls (10.167, 26.57) and (10.227, 26.563) .. (10.283, 26.545).. controls (10.339, 26.527) and (10.391, 26.499) .. (10.44, 26.467).. controls (10.548, 26.397) and (10.645, 26.306) .. (10.705, 26.193).. controls (10.735, 26.136) and (10.755, 26.074) .. (10.762, 26.01).. controls (10.768, 25.946) and (10.761, 25.881) .. (10.738, 25.821);
            \draw[knot_diagram] (10.079, 25.954).. controls (10.083, 25.984) and (10.082, 26.015) .. (10.075, 26.044);
            \draw[knot_diagram, ->] (9.205, 26.511).. controls (9.191, 26.362) and (9.213, 26.21) .. (9.269, 26.071).. controls (9.309, 25.972) and (9.366, 25.879) .. (9.446, 25.808).. controls (9.525, 25.737) and (9.628, 25.689) .. (9.735, 25.685).. controls (9.788, 25.683) and (9.842, 25.692) .. (9.891, 25.713).. controls (9.94, 25.734) and (9.984, 25.767) .. (10.018, 25.809).. controls (10.051, 25.85) and (10.073, 25.901) .. (10.079, 25.954);
            \draw[knot_diagram, <-] (10.26, 26.233).. controls (10.214, 26.237) and (10.167, 26.223) .. (10.127, 26.2).. controls (10.087, 26.178) and (10.051, 26.148) .. (10.017, 26.116).. controls (9.983, 26.085) and (9.949, 26.053) .. (9.911, 26.026).. controls (9.874, 26.0) and (9.831, 25.979) .. (9.785, 25.973).. controls (9.746, 25.968) and (9.705, 25.974) .. (9.669, 25.99).. controls (9.633, 26.006) and (9.601, 26.031) .. (9.576, 26.061).. controls (9.526, 26.122) and (9.506, 26.205) .. (9.513, 26.283).. controls (9.52, 26.364) and (9.555, 26.443) .. (9.611, 26.502);
            \draw[knot_diagram, <-] (9.973, 25.522).. controls (10.088, 25.584) and (10.208, 25.639) .. (10.304, 25.726).. controls (10.344, 25.762) and (10.379, 25.804) .. (10.405, 25.851).. controls (10.431, 25.898) and (10.446, 25.951) .. (10.446, 26.005).. controls (10.446, 26.059) and (10.43, 26.113) .. (10.397, 26.156).. controls (10.364, 26.198) and (10.313, 26.228) .. (10.26, 26.233);
            \draw[knot_diagram] (9.725, 24.763).. controls (9.663, 24.831) and (9.624, 24.92) .. (9.615, 25.012).. controls (9.607, 25.104) and (9.629, 25.198) .. (9.677, 25.276).. controls (9.745, 25.387) and (9.859, 25.461) .. (9.973, 25.522);
        \end{tikzpicture}
        $\xrightarrow{T_1}$
        \begin{tikzpicture}[scale=0.85, baseline={([yshift=-1.0ex]current bounding box.center)}]
            \draw[knot_diagram] (11.181, 25.989).. controls (11.172, 25.929) and (11.175, 25.866) .. (11.194, 25.808).. controls (11.216, 25.737) and (11.261, 25.674) .. (11.313, 25.62).. controls (11.366, 25.567) and (11.427, 25.522) .. (11.488, 25.479).. controls (11.549, 25.435) and (11.61, 25.392) .. (11.665, 25.34).. controls (11.72, 25.289) and (11.768, 25.229) .. (11.795, 25.159).. controls (11.822, 25.092) and (11.83, 25.017) .. (11.816, 24.945).. controls (11.803, 24.874) and (11.769, 24.807) .. (11.719, 24.754);
            \draw[knot_diagram, ->] (11.775, 26.029).. controls (11.792, 26.088) and (11.832, 26.141) .. (11.884, 26.173).. controls (11.923, 26.196) and (11.97, 26.209) .. (12.004, 26.239).. controls (12.032, 26.263) and (12.05, 26.296) .. (12.057, 26.331).. controls (12.063, 26.367) and (12.058, 26.404) .. (12.044, 26.437).. controls (12.029, 26.47) and (12.006, 26.499) .. (11.977, 26.521).. controls (11.949, 26.543) and (11.915, 26.558) .. (11.88, 26.566).. controls (11.81, 26.581) and (11.735, 26.569) .. (11.668, 26.542).. controls (11.602, 26.515) and (11.542, 26.474) .. (11.485, 26.429).. controls (11.389, 26.354) and (11.3, 26.267) .. (11.241, 26.161).. controls (11.212, 26.107) and (11.19, 26.049) .. (11.181, 25.989);
            \draw[knot_diagram, ->] (12.673, 26.502).. controls (12.684, 26.353) and (12.662, 26.201) .. (12.609, 26.062).. controls (12.566, 25.947) and (12.5, 25.838) .. (12.407, 25.759).. controls (12.313, 25.68) and (12.189, 25.632) .. (12.067, 25.644).. controls (12.006, 25.65) and (11.946, 25.671) .. (11.896, 25.707).. controls (11.846, 25.742) and (11.806, 25.792) .. (11.783, 25.849).. controls (11.761, 25.906) and (11.757, 25.971) .. (11.775, 26.029);
            \draw[knot_diagram, <-] (12.093, 25.964).. controls (12.132, 25.958) and (12.173, 25.965) .. (12.209, 25.981).. controls (12.245, 25.996) and (12.277, 26.021) .. (12.302, 26.052).. controls (12.352, 26.113) and (12.372, 26.195) .. (12.365, 26.274).. controls (12.358, 26.355) and (12.323, 26.433) .. (12.267, 26.492);
            \draw[knot_diagram] (11.913, 26.06).. controls (11.93, 26.045) and (11.948, 26.03) .. (11.967, 26.017).. controls (12.004, 25.99) and (12.047, 25.969) .. (12.093, 25.964);
            \draw[knot_diagram, <-] (11.905, 25.513).. controls (11.79, 25.575) and (11.67, 25.629) .. (11.574, 25.717).. controls (11.534, 25.753) and (11.499, 25.794) .. (11.473, 25.841).. controls (11.447, 25.889) and (11.432, 25.942) .. (11.432, 25.996).. controls (11.432, 26.05) and (11.448, 26.104) .. (11.481, 26.146).. controls (11.514, 26.189) and (11.565, 26.218) .. (11.618, 26.223).. controls (11.664, 26.227) and (11.71, 26.213) .. (11.751, 26.191);
            \draw[knot_diagram] (12.153, 24.754).. controls (12.215, 24.822) and (12.254, 24.911) .. (12.262, 25.002).. controls (12.271, 25.094) and (12.249, 25.188) .. (12.201, 25.267).. controls (12.133, 25.378) and (12.019, 25.451) .. (11.905, 25.513);
        \end{tikzpicture}
        \hfill
        \begin{tikzpicture}[scale=0.85, baseline={([yshift=-1.0ex]current bounding box.center)}]
            \draw[knot_diagram] (10.738, 25.821).. controls (10.71, 25.748) and (10.66, 25.685) .. (10.602, 25.632).. controls (10.544, 25.579) and (10.478, 25.536) .. (10.412, 25.493).. controls (10.346, 25.451) and (10.28, 25.408) .. (10.222, 25.356).. controls (10.163, 25.303) and (10.112, 25.241) .. (10.083, 25.169).. controls (10.055, 25.101) and (10.048, 25.026) .. (10.061, 24.955).. controls (10.075, 24.883) and (10.109, 24.816) .. (10.159, 24.763);
            \draw[knot_diagram, ->] (9.921, 26.309).. controls (9.913, 26.337) and (9.913, 26.367) .. (9.92, 26.396).. controls (9.927, 26.424) and (9.941, 26.451) .. (9.959, 26.474).. controls (9.996, 26.52) and (10.051, 26.549) .. (10.109, 26.56).. controls (10.167, 26.57) and (10.227, 26.563) .. (10.283, 26.545).. controls (10.339, 26.527) and (10.391, 26.499) .. (10.44, 26.467).. controls (10.548, 26.397) and (10.645, 26.306) .. (10.705, 26.193).. controls (10.735, 26.136) and (10.755, 26.074) .. (10.762, 26.01).. controls (10.768, 25.946) and (10.761, 25.881) .. (10.738, 25.821);
            \draw[knot_diagram, ->] (9.205, 26.511).. controls (9.191, 26.362) and (9.213, 26.21) .. (9.269, 26.071).. controls (9.309, 25.972) and (9.366, 25.879) .. (9.446, 25.808).. controls (9.525, 25.737) and (9.628, 25.689) .. (9.735, 25.685).. controls (9.788, 25.683) and (9.842, 25.692) .. (9.891, 25.713).. controls (9.94, 25.734) and (9.984, 25.767) .. (10.018, 25.809).. controls (10.051, 25.85) and (10.073, 25.901) .. (10.079, 25.954).. controls (10.083, 25.984) and (10.082, 26.015) .. (10.075, 26.044).. controls (10.07, 26.067) and (10.062, 26.089) .. (10.051, 26.109).. controls (10.032, 26.144) and (10.005, 26.174) .. (9.98, 26.205).. controls (9.955, 26.236) and (9.931, 26.27) .. (9.921, 26.309);
            \draw[knot_diagram] (9.964, 26.068).. controls (9.947, 26.054) and (9.93, 26.039) .. (9.911, 26.026).. controls (9.874, 26.0) and (9.831, 25.979) .. (9.785, 25.973).. controls (9.746, 25.968) and (9.705, 25.974) .. (9.669, 25.99).. controls (9.633, 26.006) and (9.601, 26.031) .. (9.576, 26.061).. controls (9.526, 26.122) and (9.506, 26.205) .. (9.513, 26.283).. controls (9.52, 26.364) and (9.555, 26.443) .. (9.611, 26.502);
            \draw[knot_diagram, <-] (10.26, 26.233).. controls (10.214, 26.237) and (10.167, 26.223) .. (10.127, 26.2);
            \draw[knot_diagram, <-] (9.677, 25.276).. controls (9.745, 25.387) and (9.859, 25.461) .. (9.973, 25.522).. controls (10.088, 25.584) and (10.208, 25.639) .. (10.304, 25.726).. controls (10.344, 25.762) and (10.379, 25.804) .. (10.405, 25.851).. controls (10.431, 25.898) and (10.446, 25.951) .. (10.446, 26.005).. controls (10.446, 26.059) and (10.43, 26.113) .. (10.397, 26.156).. controls (10.364, 26.198) and (10.313, 26.228) .. (10.26, 26.233);
            \draw[knot_diagram] (9.725, 24.763).. controls (9.663, 24.831) and (9.624, 24.92) .. (9.615, 25.012).. controls (9.607, 25.104) and (9.629, 25.198) .. (9.677, 25.276);
        \end{tikzpicture}
        $\xrightarrow{T_2}$
        \begin{tikzpicture}[scale=0.85, baseline={([yshift=-1.0ex]current bounding box.center)}]
            \draw[knot_diagram] (11.181, 25.989).. controls (11.172, 25.929) and (11.175, 25.866) .. (11.194, 25.808).. controls (11.216, 25.737) and (11.261, 25.674) .. (11.313, 25.62).. controls (11.366, 25.567) and (11.427, 25.522) .. (11.488, 25.479).. controls (11.549, 25.435) and (11.61, 25.392) .. (11.665, 25.34).. controls (11.72, 25.289) and (11.768, 25.229) .. (11.795, 25.159).. controls (11.822, 25.092) and (11.83, 25.017) .. (11.816, 24.945).. controls (11.803, 24.874) and (11.769, 24.807) .. (11.719, 24.754);
            \draw[knot_diagram, ->] (12.057, 26.331).. controls (12.063, 26.367) and (12.058, 26.404) .. (12.044, 26.437).. controls (12.029, 26.47) and (12.006, 26.499) .. (11.977, 26.521).. controls (11.949, 26.543) and (11.915, 26.558) .. (11.88, 26.566).. controls (11.81, 26.581) and (11.735, 26.569) .. (11.668, 26.542).. controls (11.602, 26.515) and (11.542, 26.474) .. (11.485, 26.429).. controls (11.389, 26.354) and (11.3, 26.267) .. (11.241, 26.161).. controls (11.212, 26.107) and (11.19, 26.049) .. (11.181, 25.989);
            \draw[knot_diagram, ->] (11.904, 26.183).. controls (11.938, 26.201) and (11.975, 26.214) .. (12.004, 26.239).. controls (12.032, 26.263) and (12.05, 26.296) .. (12.057, 26.331);
            \draw[knot_diagram] (12.673, 26.502).. controls (12.684, 26.353) and (12.662, 26.201) .. (12.609, 26.062).. controls (12.566, 25.947) and (12.5, 25.838) .. (12.407, 25.759).. controls (12.313, 25.68) and (12.189, 25.632) .. (12.067, 25.644).. controls (12.006, 25.65) and (11.946, 25.671) .. (11.896, 25.707).. controls (11.846, 25.742) and (11.806, 25.792) .. (11.783, 25.849).. controls (11.761, 25.906) and (11.757, 25.971) .. (11.775, 26.029).. controls (11.778, 26.04) and (11.782, 26.051) .. (11.787, 26.062);
            \draw[knot_diagram, <-] (11.618, 26.223).. controls (11.664, 26.227) and (11.71, 26.213) .. (11.751, 26.191).. controls (11.791, 26.169) and (11.827, 26.138) .. (11.861, 26.107).. controls (11.895, 26.076) and (11.929, 26.044) .. (11.967, 26.017).. controls (12.004, 25.99) and (12.047, 25.969) .. (12.093, 25.964).. controls (12.132, 25.958) and (12.173, 25.965) .. (12.209, 25.981).. controls (12.245, 25.996) and (12.277, 26.021) .. (12.302, 26.052).. controls (12.352, 26.113) and (12.372, 26.195) .. (12.365, 26.274).. controls (12.358, 26.355) and (12.323, 26.433) .. (12.267, 26.492);
            \draw[knot_diagram, <-] (11.905, 25.513).. controls (11.79, 25.575) and (11.67, 25.629) .. (11.574, 25.717).. controls (11.534, 25.753) and (11.499, 25.794) .. (11.473, 25.841).. controls (11.447, 25.889) and (11.432, 25.942) .. (11.432, 25.996).. controls (11.432, 26.05) and (11.448, 26.104) .. (11.481, 26.146).. controls (11.514, 26.189) and (11.565, 26.218) .. (11.618, 26.223);
            \draw[knot_diagram] (12.153, 24.754).. controls (12.215, 24.822) and (12.254, 24.911) .. (12.262, 25.002).. controls (12.271, 25.094) and (12.249, 25.188) .. (12.201, 25.267).. controls (12.133, 25.378) and (12.019, 25.451) .. (11.905, 25.513);
        \end{tikzpicture}
        \caption{\label{Figure:LocalOrientedMovesThree}Moves $\Omega_1, \Omega_2, T_1$ and $T_2$}
    \end{figure}
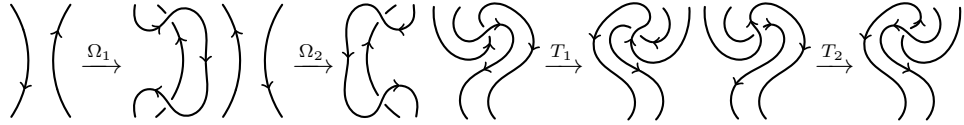

    Consider different cases where the diagram $D'$ is obtained from the diagram $D$ by one of these moves.

    \emph{$W_1^r, W_1^l, W_2^r, W_2^l, R_2^l, R_2^r$ moves}. In these cases the tensor network $T_{D', \xi'}$ differs from $T_{D, \xi}$ by inserting the inverse tensor pair ($R_{g_1, g_2}$ and $\overline{R_{g_1, g_2}}$ for the moves $W_1^r$ and $W_2^r$, $L_{g_1, g_2}$ and $\overline{L_{g_1, g_2}}$ for the moves $W_1^l$ and $W_2^l$, $Y_g$ and $\overline{Y_g}$ for the moves $R_2^l$ and $R_2^r$). So $T_{D, \xi} = T_{D', \xi'}$.

    In all these cases $r(D) = r(D')$ and $w(D) = w(D')$. Then $t_{\mathcal{S}}(D, \xi) = t_{\mathcal{S}}(D', \xi')$.

    \emph{$R_1^{+r}, R_1^{-r}, R_1^{+l}, R_1^{-l}$ moves}. In these cases the tensor network $T_{D', \xi'}$ differs from the network $T_{D, \xi}$ by inserting one of the diagrams shown in the figure \ref{Figure:Axioms1}. 
    
    Consider the case of the move $R_1^{+r}$, all others are similar. It follows from axiom 1a that $T_{D', \xi'} = AB^{-1}\cdot T_{D, \xi}$. But $r(D') = r(D) - 1$ and $w(D') = w(D) + 1$. Hence 
    \begin{multline*}
        t_{\mathcal{S}}(D', \xi') = A^{r(D')}B^{w(D')}\cdot T_{D', \xi'} = A^{r(D) - 1} B^{w(D) + 1} \cdot AB^{-1} T_{D, \xi} = \\ = A^{r(D)}B^{w(D)}\cdot T_{D, \xi} = t_{\mathcal{S}}(D, \xi).
    \end{multline*}

    \emph{$R_3$ move}. In this case the tensor network $T_{D', \xi'}$ is obtained from $T_{D, \xi}$ by replacing the fragment shown in the figure \ref{Figure:Axioms2} on the left by the fragment shown in the same figure \ref{Figure:Axioms2} on the right. By axiom 2: $T_{D', \xi'} = T_{D, \xi}$. It's clear that $r(D) = r(D')$ and $w(D) = w(D')$. So $t_{\mathcal{S}}(D, \xi) = t_{\mathcal{S}}(D', \xi')$.

    \emph{$\Omega_1$ and $\Omega_2$ moves}. Consider the case of the move $\Omega_1$, the other is similar. In this case the tensor network $T_{D', \xi'}$ is obtained from $T_{D, \xi}$ by replacing two parallel arcs by the fragment shown in the figure \ref{Figure:Axioms3} on the left. By axiom 3a: $T_{D, \xi} = T_{D', \xi'}$. As $r(D) = r(D')$ and $w(D) = w(D')$, then $t_{\mathcal{S}}(D, \xi) = t_{\mathcal{S}}(D', \xi')$.

    \emph{$T_1$ and $T_2$ moves}. Consider the case of the move $T_1$, the other is similar. In this case the tensor network $T_{D', \xi'}$ is obtained from the network $T_{D, \xi}$ by replacing the diagram shown in the figure \ref{Figure:Axioms4} on the left of the top row by the diagram shown in the same figure \ref{Figure:Axioms4} on the right of the top row. By axiom 4a: $T_{D, \xi} = T_{D', \xi'}$. As $r(D) = r(D')$ and $w(D) = w(D')$, then $t_{\mathcal{S}}(D, \xi) = t_{\mathcal{S}}(D', \xi')$.
\end{proof}

If we have the family $\{\mathcal{S}_{\alpha, \beta} | \alpha, \beta\in G\}$ of consistent tensor $G$ systems, defined for all pairs $\alpha, \beta\in G$, then it can be used to define the invariant $\tau(K)$ as follows. For each pair of values $\alpha, \beta\in G$ we can find the multi-set of values $t_{\alpha, \beta}(K) = \{t_{\mathcal{S}_{\alpha, \beta}}(D, \xi) | \xi\in Col_{\alpha, \beta}(D)\}$, where $D$ is any diagram of the link $K$. It's clear that this set $t_{\alpha, \beta}(K)$ does not depend on the diagram $D$ of the link $K$. As usual in knot theory, this multi-set can be written as an element of the integer group ring $\mathbb{Z}K$. Denote this element by the same $t_{\alpha, \beta}(K)$. Finally, we gather all these values into the formal sum $$\tau(K) = \sum\limits_{\alpha, \beta\in G}t_{\alpha, \beta}(K)\cdot (\alpha, \beta)\in (\mathbb{Z}K)(G\times G).$$

\section{Further development}

There are several directions that seem quite natural for further development.

\subsection{Study of the electric group}

The electric group has two additional generators compared to the reduced electric group. The definition of the electric group looks more general. The task is to understand whether these two additional generators are important or not. Equivalently we can ask: is it true that if $\mathcal{E}_r(K_1) = \mathcal{E}_r(K_2)$, then $\mathcal{E}(K_1) = \mathcal{E}(K_2)$?

Classical knot group is a powerful invariant of knots and links. The electric group looks simpler and weaker. The task is to understand the power of the electric group for classifying knots and links. In particular, find two different links with isomorphic electric groups.

It's known that the fundamental quandle of the link is a complete invariant (\cite{MQ, J}). Potentially any other invariant can be extracted from the fundamental quandle. The task is to find the relation between the fundamental quandle (or homomorphisms from the fundamental quandle to another finite quandle) and the electric group.

The definition of the electric group is similar to the original definition of the Alexander polynomial for knots (\cite{A}), and in particular to the definition of the Dehn presentation of the knot group. The task is to find the exact connection between the electric group and the Alexander polynomial.

The general task in studying the electric group is to understand the meaning of the electric group of links. What properties of the link does this group reflect?

\subsection{Building consistent tensor $G$-systems}

The definition of the consistent tensor $G$-system is quite abstract and contains many axioms sufficient for the invariance of $t_{\mathcal{S}}(D, \xi)$. If the group $G$ is trivial, then this system is mostly the same as the extended Yang -- Baxter operator (see \cite{TYB}). The task is to construct this system for non-trivial groups $G$.

As an easier task, we can consider a module $V$ of rank 1. In this case all tensors $Y_g, L_{g_1, g_2}, R_{g_1, g}$ and their inverse are just numbers, and tensor contraction is just multiplication of these numbers. So the additive structure of the ring $K$ is not important. Instead of $K$ with multiplication, we can consider an abelian group and write tensor operations additively. Then all axioms for a consistent tensor $G$-system become a system of linear equations. The task is to find a solution of these equations for some groups.

\end{document}